\documentclass[12pt,reqno]{amsart}
\usepackage{cmap}
\usepackage[T1]{fontenc}
\usepackage{amsmath,amsfonts,amssymb,amsthm,mathtools,bbold,adjustbox}
\usepackage[vcentermath]{youngtab}
\usepackage[all]{xy}
\usepackage{csquotes}
\usepackage[usenames,dvipsnames]{xcolor}
\usepackage{pstricks}
\usepackage{graphicx,tikz-cd}
\usepackage{rotating}
\usepackage{cite}
\usepackage{setspace}
\usepackage{tabu}
\usepackage{tikz}
\usetikzlibrary{decorations.markings, patterns, chains}
\usepackage[margin=1in]{geometry}
\usepackage{siunitx} 
\usepackage{ytableau}

\DeclareMathOperator{\Hom}{Hom}

\DeclareMathOperator{\Irr}{Irr}
\DeclareMathOperator{\Tr}{T}
\DeclareMathOperator{\MD}{mod}
\DeclareMathOperator{\aff}{aff}
\DeclareMathOperator{\rev}{rev}
\DeclareMathOperator{\dom}{dom}

\DeclareMathOperator{\LKC}{LKC}
\DeclareMathOperator{\RKC}{RKC}

\DeclareMathOperator{\charge}{charge}
\DeclareMathOperator{\lch}{lch}
\DeclareMathOperator{\Diff}{Diff}

\DeclarePairedDelimiter\floor{\lfloor}{\rfloor}

\newcommand{\mcC}{\mathcal{C}}

\newcommand{\mcO}{\mathcal{O}}

\newcommand{\mcF}{\mathcal{F}}

\newcommand{\mbZ}{\mathbb{Z}}

\newcommand{\mfF}{\mathfrak{F}}

\theoremstyle{plain}
\newtheorem{Theorem}{Theorem}[section]
\newtheorem{Proposition}[Theorem]{Proposition}
\newtheorem{Lemma}[Theorem]{Lemma}
\newtheorem{Corollary}[Theorem]{Corollary}
\newtheorem{Conjecture}[Theorem]{Conjecture}
\newtheorem{Example}[Theorem]{Example}

\theoremstyle{definition}
\newtheorem{Definition}[Theorem]{Definition}

\theoremstyle{remark}
\newtheorem{Remark}[Theorem]{Remark}

\author{Pablo Boixeda Alvarez}
\address{Department of Mathematics\\ Massachusetts Institute of Technology\\ Cambridge, MA, 02139, USA}
\email{pabloboixeda@hotmail.com}

\author{Li Ying}
\address{Department of Applied and Computational Mathematics and Statistics\\ University of Notre Dame\\ Notre Dame, IN 46556, USA}
\email{98yingli@gmail.com}

\author{Guangyi Yue}
\address{Department of Mathematics\\ Massachusetts Institute of Technology\\ Cambridge, MA 02139, USA}
\email{guangyiyue@hotmail.com}

\begin{document}
\setcounter{tocdepth}{1}

\title{Affine Springer fibers and the affine matrix ball construction for rectangular type nilpotents}
\maketitle
\begin{abstract}
In this paper, we study the affine Springer fiber $\mcF l_N$ in type $A$ for rectangular type semisimple nil-element $N$ and calculate the relative position between irreducible components. In particular, we use the affine matrix ball construction to show the relative position map is compatible with the Kazhdan-Lusztig cell structure, generalizing the work of Steinberg and van Leeuwen.
\end{abstract}
\tableofcontents

\pagestyle{plain}

\section{Introduction}
In the paper \cite{kazhdan1979representations}, Kazhdan and Lusztig laid a foundation for studying representations of Hecke algebras, and in particular they introduced the notion of (two-sided, left, right) cells for Coxeter groups. In type $A$, the Kazhdan-Lusztig cell structure of the symmetric group corresponds to the well known Robinson-Schensted correspondence, which is a bijection between the symmetric group $S_n$ and pairs of standard Young tableau of the same shape $\lambda\vdash n$:
$$w\in S_n\mapsto (\text{insertion tableau } P,\text{recording tableau }Q).$$ 
Namely,
\begin{enumerate}
	\item two permutations are in the same two-sided cell iff they have the same associated partition $\lambda$;
	\item two permutations are in the same right cell iff they have the same insertion tableau $P$;
	\item two permutations are in the same left cell iff they have the same recording tableau $Q$.
\end{enumerate}

Robinson-Schensted correspondence are realized by many equivalent combinatorial algorithms, for example the row-insertion algorithm and the matrix ball construction \cite{viennot1977forme,fulton1997young}. This combinatorial correspondence appears in the study of (finite) Springer fibers. Given a nilpotent $N$ of type $\lambda$, Spaltenstein\cite{Spaltenstein} labeled the irreducible components of the Springer fiber of $N$ by standard Young tableau of shape $\lambda$. Later on, Steinberg \cite{Steinberg} showed the relative position between two components labeled by tableaux $P$ and $Q$ respectively are exactly the permutation corresponding to $(P,Q)$ under the Robinson-Schensted algorithm. This result is compatible with the cell structure since the image of the relative position map is exactly a right (resp. left) cell if we fix the first (resp. second) component. These nice interpretations are further extended by van Leeuwen in \cite{vanLeeuwen}. And the natural question is to find an analogue in the affine setting. 

On the combinatorial side, the Robinson-Schensted correspondence is generalized by Shi \cite{shi2006kazhdan} to the affine symmetric group $\widetilde{S_n}$, giving a parametrization of the left cells by tabloids. The shape of these tabloids determines the two-sided cell. Later Honeywill \cite{honeywill2005combinatorics} added the third piece of data, weights, to make it a bijection:
$$w\in\widetilde{S_n}\mapsto ( \text{insertion tabloid } P,\text{recording tabloid }Q,\text{dominant weight }\rho).$$

Both Shi and Honeywill's algorithms are very involved and Chmutov, Pylyavskyy, Yudovina \cite{chmutov2018matrix} generalized the matrix ball construction given by Viennot to give a simpler and more intuitive realization. This generalized algorithm, named the affine matrix ball construction, has a variety of nice applications. In particular, it is used to understood the structure of bi-directed edges in the Kazhdan-Lusztig cells in affine type $A$ in \cite{chmutov2017monodromy}. Most importantly, fibers of the inverse map of affine matrix ball construction possess a Weyl group symmetry. The relative position map in the affine setting, though not injective, is proven to have the same property, which is our motivation to establish a bijection similar to Steinberg's.

On the geometry side, the affine Springer fibers appearing in this paper have been studied before. In particular the geometry of these are studied in \cite{GKM}. The case for type $(1^n)$ has been studied in further depth. In particular the cohomology has been studied by works of Goresky, Kottwitz, Macpherson \cite{GKM2}, Hikita \cite{Hi} and Kivinen \cite{Kivinen}. This affine Springer fiber is also related with the representation theory of small quantum groups as proven in upcoming work of Bezrukavnikov, McBreen and upcoming jont work of the first author with Bezrukavnikov, Shan and Vasserot \cite{BBASV}.

In \cite{LusztigConjclass}, Lusztig introduced the partitions of the extended affine Weyl group $\widetilde{W}$ into $S$-cells and $\tilde{S}$-cells, both parametrized by the conjugacy classes in the finte Weyl group $W$, and conjectured that the image of the relative position map from pairs of irreducible components of the affine Springer fiber to the extended affine Weyl group
\begin{equation*}
\label{Lusztigconj}
\Irr(\mcF l_N)\times\Irr(\mcF l_N)\rightarrow \widetilde{W}
\end{equation*}
is exactly the $S$-cell of type $\gamma$ where $N$ is a regular semi-simple nil-element of type $\gamma$. Lusztig's conjecture is proved in the recent work of Finkelberg, Kazhdan and Varshavsky \cite{finkelberg2020lusztig} in general type. They use the affine Springer resolution and families of Springer fibers, but here we focus on a single affine Springer fiber. Also Lawton \cite{Lawton} showed that in type $\tilde{A}$, the two-sided cells and the $\tilde{S}$-cells coincide.

It follows directly from \cite{finkelberg2020lusztig} and results in this paper that the $S$-cells and the two-sided Kazhdan-Lusztig cells agree in type $\tilde{A}$ and for rectangular type, which also coincides with $\tilde{S}$-cells by \cite{Lawton}. The general relationship between $S$-cells and two-sided cells is still unknown, and we leave this for future investigation.

In this paper, we focus on type $A$ and $N$ is of rectangular type $(l^m)$, and study the relationship between relative position and the two-sided, left and right cells instead of $S$-cells, giving an affine generalization of Spaltenstein, Steinberg and van Leeuven's results in the finite case. The special column type $(1^n)$ is presented in \cite{boixeda2019fix,yuganyou,Boixedathesis}. Namely, we use affine matrix ball contruction to establish a bijection between pairs of irreducible components modulo common translations with $\Omega_{(l^m)}$, which are triples $(P,Q,\rho)$ of rectangular-type and $\rho$ is not necessarily dominant. This is given in the following commutative diagram:
\begin{equation*}
\begin{tikzcd}
\Irr(\mathcal{F}l_{N})/\Lambda\arrow[dd,"\theta"] && \Irr(\mathcal{F}l_{N})\times_\Lambda \Irr(\mathcal{F}l_{N})\arrow[ll, "pr_i"']\arrow[rr, "r"]\arrow[dd,"\Theta"] && \widetilde{S_n}\\
&&&&\\
T(l^m) && \Omega_{(l^m)}\arrow[ll,"pr_i"']\arrow[uurr,"\Psi"'] &&
\end{tikzcd}
\end{equation*}
where $T(l^m)$ is the collection of tabloids of shape $(l^m)$, $\Psi$ is the inverse of the affine matrix ball construction, and $pr_i$, $i=1,2$, are the projections onto the first and second component respectively. It follows that the image of the relative position map $r$ is the two-sided cell of type $(l^m)$. Moreover, $r(pr_1^{-1}(C))$ (resp. $r(pr_2^{-1}(C))$) is a right (resp. left) cell for any $C \in \Irr(\mcF l_N)/\Lambda$, similar to the finite scenario.

The rest of the paper is organized as follows. In Section 2, we review the affine matrix ball construction and the related combinatorics about the affine symmetric group. And in Section 3 we study the explicit structure of two-sided cell of rectangular type. Section 4 is a review of basics on affine Springer fibers and in Section 5 we study the geometry of the irreducible components of $\mcF l_N$ when $N$ is of rectangular type and calculate the relative position between any two irreducible components. Section 6 deals with the case of $n=2$ explicitly. The proof of the main theorem is presented in Section 7. In the appendix, we give diagrams of left Knuth classes containing $w_0^\lambda$ when $\lambda=(2,2),(3,3)$ and $(2,2,2)$.



{\bf Acknowledgements.} The authors would like to thank Roman Bezrukavnikov for suggesting this problem and continuous discussions throughout the process. Also, the authors are grateful to Zhiwei Yun for many useful discussions. The first author also wants to thank Dongkwan Kim and Pavlo Pylyavskyy for a useful early discussion.

\section{Combinatorial preliminaries}

For the entire paper, we fix a positive integer $n$. Denote $[a,b]=[a,a+1,\ldots,b] $ for any $a,b\in\mbZ, a<b$ and $[a]=[1,a]$ for $a\in\mathbb{Z}_{>0}$. For any $i\in\mbZ$, let $\overline{i}$ be the residue class $i+n\mbZ$, and denote $\left[\overline{n}\right]=\{\overline{1},\ldots,\overline{n}\}.$
\subsection{Affine Symmetric Group}
Let $S_n$ be the \textit{symmetric group} on $n$ letters, which is the Weyl group of type $A_{n-1}$. The \textit{extended affine symmetric group} $\overline{S_n}$ is the collection of all bijections $w: \mathbb{Z}\rightarrow\mathbb{Z}$ satisfying $w(i+n)=w(i)+n$ for all $i\in\mbZ.$ And we call the elements in $\overline{S_n}$ to be \textit{extended affine permutations}. Let $\widetilde{S_n}\subset\overline{S_n}$ be the \textit{affine symmetric group} consisting of all $w\in\overline{S_n}$ satisfying $\sum_{i=1}^nw(i)=\frac{n(n+1)}{2}$, and elements inside $\widetilde{S_n}$ are called \textit{affine permutations}. The (extended) affine symmetric group is exactly the (extended) affine Weyl group of type $\widetilde{A_{n-1}}.$ 

Since (extended) affine permutations are determined by its values on $[n]$, we use the \textit{window notation} $[w(1),\ldots,w(n)]$ to represent $w$. We denote $\overline{w}=\left[\overline{w(1)},\ldots,\overline{w(n)}\right]\in S_n$.

The affine symmetric group $\widetilde{S_n}$ is the Coxeter group generated by simple reflections $s_1,\ldots,s_{n-1},s_0=s_n$ (we take the indices $i$ in $[\overline{n}]$ without ambiguity) under Coxeter relations, where $s_i$ can be viewed as a permutation on $\mathbb{Z}$ such that
\[   s_i(x)=\left\{
\begin{array}{ll}
x+1, & x\equiv i\;(\MD n), \\
x-1, & x\equiv i+1\;(\MD n),  \\
x, & \text{else.}\\
\end{array} 
\right. \]
And $\overline{S_n}=\Omega \ltimes\widetilde{S_n}$ where $\Omega$ is the infinite cyclic group generated by $s=[2,3,\ldots,n+1]$. The \textit{rotation map} $\phi(w)=sws^{-1}$ is an automorphism of $\widetilde{S_n}$ sending $ s_{\overline{i}}$ to $s_{\overline{i+1}}$ for $\overline{i}\in[\overline{n}]$, which corresponds to the rotation of the Dynkin diagram of type $\widetilde{A_{n-1}}$.

There are two well-known formulas for computing the length of an affine permutation, the first one is given by Shi \cite{shi2006kazhdan}:
\begin{Lemma}
	\label{lengthformula}
	For $w\in\widetilde{S_n}$, we have
	\begin{equation*}
	\begin{split}
	\ell(w)&=\sum_{1\le i<j\le n}\left|\floor*{\frac{w(j)-w(i)}{n}}\right|\\
	&=\#\left\{(i,j)\in[n]\times\mathbb{Z}_{>0}\mid i<j, w(i)>w(j)\right\}.
	\end{split}
	\end{equation*}
\end{Lemma}

\subsection{Kazhdan-Lusztig Cells and Affine Matrix Ball Construction}
We now follow \cite{chmutov2017monodromy,chmutov2018matrix} and identify affine permutations with its matrix ball configuration. In detail, for $w\in\widetilde{S_n}$, we draw a $\mathbb{Z}\times\mathbb{Z}$ matrix with row labels increasing southwards and column labels increasing eastwards. If $w(i)=j$, we draw a ball in the $(i,j)$-position of the matrix and will be named also by $(i,j)$ without ambiguity. And we denote $\mathcal{B}_w=\{(i,w(i))\mid i\in\mathbb{Z}\}$ which is the collection of the balls of $w$. The periodicity of $w$ implies that $(i,j)\in \mathcal{B}_w$ iff $(i+n,j+n)\in \mathcal{B}_w$. We say $(i+kn,j+kn)$ for $k\in\mathbb{Z}$ are the $(n,n)$-\textit{translates} of $(i,j)$. For two balls $(i,j), (k,l)\in\mathcal{B}_w$, we define the southeast (partial) ordering $\le_{SE}$ by $(i,j)\le_{SE}(k,l)$ iff $i\ge k$ and $j\ge l$, i.e. $(i,j)$ is southeast of $(k,l)$. Other relations using compass directions can be defined similarly, and are also partial orders on $\mathbb{Z}\times\mathbb{Z}$.

A \textit{partition} $\lambda$ of size $n\in\mathbb{N}$ is a finite tuple of weakly decreasing positive integers $\lambda=(\lambda_1,...,\lambda_k)$ with sum $n$. Denote $\ell(\lambda) = k$ to be the number of nonzero parts of $\lambda$. The \textit{Young diagram} of a given partition $\lambda$ is a left-justified collection of boxes with the first row having $\lambda_1$ boxes, second row having $\lambda_2$ boxes and so on. And we denote $\lambda^{\Tr}$ to be the \textit{transpose} of $\lambda$.

For a given partition $\lambda$ of size $n$, a \textit{tabloid} of shape $\lambda$ is an equivalence class of bijective fillings of the Young diagram of $\lambda$ with $\left[\overline{n}\right]$, such that two fillings are equivalent if one is obtained from the other by permuting the entries of each row. 

We denote the collection of all tabloids of shape $\lambda$ to be $T(\lambda)$. And let $T^\lambda\in T(\lambda)$ be the tabloid with $\overline{1}$ in the first row, $\overline{2}$ in the second row,..., $\overline{\lambda^{\Tr}_1}$ in the last row, $\overline{\lambda^{\Tr}_1+1}$ in the first row and so on. There is a natural left action of $S_n$ on $T(\lambda)$ and for $X\in T(\lambda)$, $X+k$  is defined to be the tabloid adding $k$ to each entry in $X$.

For any tabloid $X\in T(\lambda)$, $i\in[\lambda^{\Tr}_1]$, let $X_i\subset\left[\overline{n}\right]$ be the $i$-th row of $X$. Denote $X_i=\left\{\overline{X_{i,1}},\overline{X_{i,2}},\ldots,\overline{X_{i,\lambda_i}}\right\}$ such that $X_{i,1},\ldots,X_{i,\lambda_i}\in[n]$ and $X_{i,1}<\ldots<X_{i,\lambda_i}$. Moreover, throughout the paper, we always extend the column indices as:
$$X_{i,j+k\lambda_i}=X_{i,j}+kn$$
for $i\in[\lambda^{\Tr}_1]$ and $k\in\mbZ.$

\begin{Example}
	For $\lambda=(3,2,2,1)$, $T^\lambda$ and a tabloid $X$ in $T(\lambda)$ are the following: $$T^{\lambda}=\,\raisebox{21pt}{\begin{ytableau}\overline{1}& \overline{8} & \overline{5} \\\overline{2}  & \overline{6} \\\overline{7}  & \overline{3 }\\ \overline{4}  \end{ytableau}},\;X=\,\raisebox{21pt}{\begin{ytableau}\overline{7}& \overline{1} & \overline{4} \\\overline{2}  & \overline{5} \\\overline{6}  & \overline{3 }\\ \overline{8}  \end{ytableau}}.$$
	Then $X_1=\left\{\overline{7}, \overline{1} , \overline{4}\right\}$, $X_3=\left\{\overline{3},\overline{6}\right\}\subset\left[\overline{8}\right]$ and $X_{3,1}=3$, $X_{3,2}=6,$ $X_{3,3}=11$, $X_{4,2}=16.$
\end{Example}

\begin{Definition}
	\label{streamchannel}
	Given an affine permutation $w$, a subset $C\subset\mathcal{B}_w$ is called a \textit{stream} if it is invariant under $(n,n)$-translations and forms a chain under the southeast partial ordering $\le_{SE}$. The number of distinct $(n,n)$-translation classes of a stream $C$ is called the \textit{density} of $C$. A subset $C'\subset\mathcal{B}_w$ is called a \textit{anti-stream} if it forms a chain under the southwest partial ordering $\le_{SW}$. The number of entries in an anti-stream $C'$ (which is always finite) is called the \textit{density} of $C'$. 
	
\end{Definition}
From Lusztig we could associate $\lambda(w)=(d_1,d_2-d_1,d_3-d_2,\ldots)$ to any affine permutation $w$ where $d_i$ is the maximal one among the sums of densities of $i$ disjoint streams in $\mathcal{B}_w$.
\cite[Theorem 1.5]{greene1976some} guarantees that $\lambda(w)$ is a partition and is called \textit{the partition associated to} $w$. Moreover, $(d'_1,d'_1-d'_2,\ldots)$, where $d'_i$ is the maximal one among the sums of densities of $i$ disjoint anti-streams in $\mathcal{B}_w$, is the partition $\lambda(w)^{\Tr}$. 

 To determine the cell structure, we recall the affine matrix ball construction $\Phi$ from \cite{chmutov2018matrix}, which is a nice generalization of Viennot's geometric construction \cite{viennot1977forme} of the classical Robinson-Schensted correspondence.
\begin{align*}
\Phi:\widetilde{S_n} & \rightarrow \Omega=\bigsqcup_{\lambda\vdash n}\left\{(P,Q,\rho)\,\left|\, P, Q\in T(\lambda),\rho\in\mathbb{Z}^{\ell(\lambda)},\sum_{i=1}^{\ell(\lambda)}\rho_i=0\right.\right\}\\
w & \mapsto (P(w),Q(w),\rho(w)).
\end{align*}
We call $P(w),Q(w),\rho(w)$ to be the \textit{insertion tabloid}, \textit{recording tabloid} and \textit{weight} of $w$ respectively. And the common shape of $P(w)$ and $Q(w)$ is the \textit{associated partition} $\lambda(w)$. These satistics give the structure the Kazhdan-Lusztig cells:
\begin{Theorem}[\cite{lusztig1985two,shi2006kazhdan,jian1991generalized,chmutov2018matrix}]\label{cell:thm}
	\begin{enumerate}
		\item Two affine permutations are in the same two-sided cell iff they have the same associated partition.
		\item Two affine permutations are in the same left (resp. right) cell iff they have the same recording (resp. insertion) tabloid.
	\end{enumerate}
\end{Theorem}
We denote $C_\lambda$ to be the two-sided cell containing affine permutations with associated patition $\lambda$, and $L_X$ (resp. $R_X$) to be the left (resp. right) cell containing affine permutations with recording (resp. insertion) tabloid $X$. In particular, the longest element $$w_0^\lambda=[\lambda^{\Tr}_1,\lambda^{\Tr}_1-1,\ldots,1,\lambda^{\Tr}_1+\lambda^{\Tr}_2,\lambda^{\Tr}_1+\lambda^{\Tr}_2-1,\ldots,\lambda^{\Tr}_1+1,\ldots,n,n-1\ldots,n-\lambda^{\Tr}_{\lambda_1}+1]$$ in $ S_{\lambda^{\Tr}_1}\times S_{\lambda^{\Tr}_2}\times\ldots \times S_{\lambda^{\Tr}_{\lambda_1}} \hookrightarrow\widetilde{S_n}$ is contained in $ L_{T^\lambda}\cap R_{T^\lambda}$.

We refer the readers to \cite{chmutov2018matrix} for the details of the algorithm of computing $\Phi$ and only state the formula in a very special case which we will use later.

\begin{Lemma}\label{lemAMBC}
	Suppose $n=ml$ for some $m,l\in\mbZ_{>0}$ and $w\in\widetilde{S_n}$ satisfies the following two conditions:
	\begin{enumerate}
		\item $w(i)<w(i+m)<\ldots<w(i+m(l-1))<w(i)+n$ for $i\in[m]$;
		\item $w(1+(j-1)m)>w(2+(j-1)m)>\ldots>w(jm)$ for $j\in[l]$.
	\end{enumerate}
Then $\lambda(w)=(l^m)$ and $\Phi(w)=(P,T^\lambda,\rho)$ where
\begin{align*}
P_i = & \left\{\overline{w(m+1-i)},\overline{w(2m+1-i)},\ldots,\overline{w(n+1-i)}\right\},\\
\rho_i= &  \sum_{j=1}^{l}\left\lceil\frac{w(jm+1-i)}{n}\right\rceil-l,
\end{align*}
 for $i\in[m]$.

\end{Lemma}

Now we describe the Weyl group symmetry of the fiber of the inverse of $\Phi$.
\begin{Definition}[\cite{chmutov2017monodromy}]\label{deflch}
	Suppose $\lambda_i=\lambda_{i+1}$. For $X\in T(\lambda)$, the \textit{local charge} $\lch_i(X)$ at row $i$ is defined to be the  smallest integer satisfying $X_{i+1,j+\lch_i(X)}\ge X_{i,j}$ for all $j\in\mbZ$. And the \textit{charge matching} at row $i$ is the map $\{X_{i,j}\mid j\in\mbZ\}\rightarrow\{X_{i+1,j}\mid j\in\mbZ\}$ via $X_{i,j}\mapsto X_{i+1,j+\lch_i(X)}.$ If $\lambda_i>\lambda_{i+1}$, we define $\lch_i(X)=0$.
\end{Definition}

\begin{Definition}
	Given $X\in T(\lambda)$, the \textit{symmetrized offset constant} $s(X)\in \mbZ^{\ell(\lambda)}$ of $X$ is defined as:
	$$s_i(X)=\sum_{j=i'}^{i-1}\text{lch}_j(X),$$
	where $i'$ is the first row in $\lambda$ with length $\lambda_i$. The \textit{charge} of $X$ is defined to be:
	$$\charge(X)=\sum_{i=1}^{\ell(\lambda)-1}i\cdot\text{lch}_i(X).$$
\end{Definition}

\begin{Definition}
	The weight $\rho$ in the triple $(P,Q,\rho)$ is called \textit{dominant}, if $\rho-{s}(P)+{s}(Q)$ is increasing segmentwise according to the part sizes of $\lambda$, i.e., for each $i$, either $\lambda_i>\lambda_{i+1}$, or $\lambda_i=\lambda_{i+1}$ and $(\rho-{s}(P)+{s}(Q))_i<(\rho-{s}(P)+{s}(Q))_{i+1}.$ And we define $$\Omega_{\text{dom}}:=\bigsqcup_{\lambda\vdash n}\left\{(P,Q,\rho)\in \Omega\mid \rho\text{ is dominant in }(P,Q,\rho)\right\}.$$
	The \textit{dominant representative} $\rho'$ of $\rho$ in the triple $(P,Q,\rho)$ can be computed by
	$$\rho'=(\rho-{s}(P)+{s}(Q))^{\dom}+{s}(P)-{s}(Q),$$
	where $(\rho-{s}(P)+{s}(Q))^{\dom}$ is the segmentwise increasing rearrangement of $\rho-{s}(P)+{s}(Q)$ according to part sizes of $\lambda.$
\end{Definition}

In fact, $\Phi$ is a bijection between $\widetilde{S_n}$ and $\Omega_{\dom}$, and its inverse can be extended to $$\Psi:\Omega\rightarrow\widetilde{S_n}.$$ We refer readers to \cite{chmutov2018matrix} for details of $\Psi$ and only point out the following crucial result:
\begin{Theorem}[\cite{chmutov2018matrix}]
	\label{thminverseAMBC}
	For any $w\in\widetilde{S_n}$, $\Psi(\Phi(w))=w,$ and for any triple $(P,Q,\rho)\in \Omega$, we have $\Phi(\Psi(P,Q,\rho))=(P,Q,\rho')$, where $\rho'$ is the dominant representative of $\rho.$
\end{Theorem}

The inverse of permutations behaves nicely under affine matrix ball construction:
\begin{Proposition}\label{propinvAMBC}
	For $w\in\widetilde{S_n}$, $\Phi(w^{-1})=(Q(w),P(w),(-\rho(w))')$ where $(-\rho(w))'$ is the dominant representative of $-\rho(w)$ in the fiber (of $\Psi$).
\end{Proposition}

We define \begin{equation*}\widetilde{\rho}(w)=\rho(w)-{s}(P(w))+{s}(Q(w)),
\end{equation*}
to be the \textit{centralized weight} of $w$, which is a segmentwise increasing vector according to the part sizes of $\lambda$, and the above proposition is equivalent to saying \begin{equation*}
\widetilde{\rho}(w)=-\widetilde{\rho}(w^{-1})^{\text{s.}\rev},
\end{equation*}
where $\widetilde{\rho}(w^{-1})^{\text{s.}\rev}$ is the segmentwise reverse of the vector $\widetilde{\rho}(w^{-1})$.

We end this subsection with the following result on how rotation interacts with affine matrix ball construction.
\begin{Lemma}
	\label{rotationAMBC}
	For any $w\in\widetilde{S_n}$ and $k\in[0,n-1]$, there is $$\Phi\left(\phi^k(w)\right)=\left(P(w)+k,Q(w)+k,\rho(w)+\delta^k\left(P(w)\right)-\delta^k\left(Q(w)\right)\right)$$
	where 
	$$ \delta^k_i\left(X\right)=\sum_{j=1}^{\lambda_i(w)}\mathbb{1}_{[n-k+1,n]}\left(X_{i,j}\right),\quad i\in[\ell(\lambda(w))].$$
	
\end{Lemma}
\begin{proof}
	The matrix balls of $\phi^k(w)$ come from that of $w$ by shifting southwestwards by $(k,k)$. The relative positions of the matrix balls do not change, so do the numberings at each step of $\Phi$ in \cite{chmutov2018matrix}. Therefore the coordinate of the back corner posts at each step of the algorithm will shift southeastwards by $(k,k)$ as well. Hence $P(\phi^k(w))=P(w)+k$, $Q(\phi^k(w))=Q(w)+k$, and
	\begin{equation*}
	\rho_i(\phi^k(w))=\rho_i(w)+\sum_{j=1}^{\lambda_i(w)}\left\{
	\begin{array}{lll}
	1, & \mbox{if $Q_{i,j}(w)\in [n-k],P_{i,j}(w)\in[n-k+1,n]$}\,;\\
	-1, & \mbox{if $P_{i,j}(w)\in [n-k],Q_{i,j}(w)\in[n-k+1,n]$}\,;\\
	0, & \text{otherwise.}
	\end{array}\right.
	\end{equation*}
	
\end{proof}

\subsection{Knuth Equivalence Classes}
First we define decent sets for both permutations and tabloids. 
\begin{Definition}
	Given an affine permutation $w$, we define its \textit{right descent set} $R(w)$ and \textit{left descent set} $L(w)$ as:
	\begin{align*}
	R(w) = & \left\{\overline{i}\in[\overline{n}]\mid w(i)>w(i+1)\right\},\\
	L(w) = & \left\{\overline{i}\in[\overline{n}]\mid w^{-1}(i)>w^{-1}(i+1)\right\}.
	\end{align*}
	Given a tabloid $X$, the $\tau$-\textit{invariant} of $X$ is defined as:
	\begin{equation*}
	\tau(X)=\left\{\overline{i}\in[\overline{n}]\mid \overline{i} \text{ lies in a strictly higher row than }\overline{i+1}\text{ in }X\right\}.
	\end{equation*}
\end{Definition}
Decent sets interact nicely with affine matrix ball construction:
\begin{Proposition}[\cite{chmutov2017monodromy}]\label{propdecent}
	For $w\in\widetilde{S_n}$, $L(w)=\tau(P(w))$ and $R(w)=\tau(Q(w))$.
\end{Proposition}
Now we are able to define Knuth moves.
\begin{Definition}
	Two affine permutations $w$ and $ws_i$ (resp. $s_iw$) are connected by a \textit{right (resp. left) Knuth move} at position $\overline{i}$ if $R(w)$ and $R(ws_i)$ (resp. $L(w)$ and $L(s_iw)$) are incomparable under the containment partial ordering. We therefore have an equivalence class named \textit{right (resp. left) Knuth class} generated by right (resp. left) Knuth moves $w\sim_{\text{RKC}} ws_i$ (resp. $w\sim_{\text{LKC}} ws_i$), and we denote $\text{RKC}_w$ (resp. $\text{LKC}_w$) to be the right (resp. left) Knuth class containing $w$.
	
	Two tabloids $X$ and $X'$ are connected by a \textit{Knuth move} if for some $i$, $X$ is obtained from $X'$ by interchanging $\overline{i}$ and $\overline{i+1}$ and $\tau(X)$ and $\tau(X')$ are incomparable. We call an equivalence class generated by Knuth moves, a \textit{Knuth class}.
	
\end{Definition}
The following remarkable theorem describes how the image of an affine permutations under affine matrix ball construction behaves after a Knuth move (we state the left Knuth version for convenience of the calculations later).
\begin{Theorem}[\cite{chmutov2017monodromy}]
	\label{thmknuth}
	Suppose $w\sim_{\LKC}s_kw$, then:
		\begin{enumerate}
			\item $Q(w)=Q(s_kw)$;
			\item $P(s_kw)$ differs from $P(w)$ by a Knuth move exchanging $\overline{i}$ and $\overline{i+1}$ for some $i\in\{k-1,k,k+1\}$;
			\item $\rho(s_kw)=\rho(w)$ if $\overline{i}\neq \overline{n}$, otherwise $\rho(s_kw)$ differs by $\rho(w)$ by subtracting 1 from row $k'$ and adding 1 to row $k$, where $\overline{i}=\overline{n}$ lies in row $k$ in $P(w)$ and $\overline{i+1}=\overline{1}$ lies in row $k'$ in $P(w)$.
		\end{enumerate}
\end{Theorem}

Theorem \ref{thmknuth} tells a right (resp. left) cell is a disjoint union of right (resp. left) Knuth classes. And \cite{chmutov2017monodromy} gives a complete characterization of Knuth classes by specifying what $Q$ and $\rho$ could be like in each Knuth class. Denote $d_\lambda=\gcd(\lambda^{\Tr}_1,\lambda^{\Tr}_2,\ldots)$.
\begin{Theorem}[\cite{chmutov2017monodromy}, Theorem 8.6]
	\label{knuthQ}
	Let $X,X'\in T(\lambda)$, then $X$ and $X'$ are in the same Knuth class iff $$\charge(X)\equiv\charge(X')\quad(\MD d_\lambda).$$
\end{Theorem}
\begin{Definition}
	Given $w\in\widetilde{S_n}$, we define the \textit{monodromy group} $G^R_w$ based at $w$ to be $$G^R_w=\{\rho(w')-\rho(w)\mid w'\in\text{RKC}_{w}, Q(w')=Q(w)\}.$$
\end{Definition}
\begin{Theorem}[\cite{chmutov2017monodromy}, Theorem 7.28]
	\label{knuthrho}
	For any $w$ with associated partition $\lambda$, we have
	$$G^R_w=\left\{\left.\sum_{i=1}^ka_i\mathbb{1}_{m_i}\right| a_i\in\mbZ ,\sum_{i=1}^k a_im_i=0\right\}$$
	where $m_1>m_2>\ldots>m_k$ are distinct part sizes of $\lambda^{\Tr}$ and $\mathbb{1}_{m_i}\in\mbZ^{\ell(\lambda)}$ with 1's in the first $m_i$ rows and 0 after on. In particular, when $\lambda$ is a rectangle, the monodromy group is trivial.
\end{Theorem}
Theorem \ref{knuthrho} indicates the right (resp. left) Knuth classes of rectangular type permutations are finite, which is the foundation of our study.

\begin{Lemma}\label{Porbit:combinatorial}
	Suppose $w\in\widetilde{S_n}$ satisfies
	$$w(i)<w(i+m)<\ldots<w(i+m(l-1))<w(i)+n,\quad\forall i\in[m]$$ and $u\in (S_m)^l$ is the unique permutation such that
	$$wu(1+(j-1)m)<wu(2+(j-1)m)<\ldots<wu(jm),\quad\forall j\in[l].$$ Then $$wu(i)<wu(i+m)<\ldots<wu(i+m(l-1))<wu(i)+n,\quad\forall i\in[m].$$
\end{Lemma}
\begin{proof}
	Fix any $j\in[0,l-1]$, and denote $a_i=w(i+jm)$, and $b_i=w(i+(j+1)m)$ for $i\in[m]$. Let $a_{\sigma(1)}<a_{\sigma(2)}<\ldots<a_{\sigma(m)}$ and $b_{\eta(1)}<b_{\eta(2)}<\ldots<b_{\eta(m)}$ for some $\sigma,\eta\in S_m$. It suffices to show that $a_{\sigma(i)}<b_{\eta(i)}$ for all $i\in[m]$ and this reduces to showing there are at most $i-1$ elements in $\{b_1,\ldots,b_m\}$ that are smaller than $a_{\sigma(i)}$. This is true because we know $b_{\sigma(i)},b_{\sigma(i+1)},\ldots,b_{\sigma(m)}$ are larger than $a_{\sigma(i)}$.
\end{proof}
\section{Structure of two-sided Kazhdan-Lusztig cells of rectangular type}
In this section we restrict ourselves to $n\ge3$ and study the structure of left (resp. right) Kazhdan-Lusztig cells and left (resp. right) Knuth classes inside the two sided cell $C_{\lambda}$ when $\lambda=(l^m)$ ($n=lm$) is a rectangle. The case of $n=2$ is addressed separately in Section 6.
\begin{Lemma}\label{equidist}
	For any partition $\lambda$, let $d_\lambda=\gcd(\lambda_1^{\Tr},\lambda_2^{\Tr},\ldots)$. Then
	\begin{enumerate}
		\item $\charge(\phi(X))\equiv \charge(X)-1\;(\MD d_\lambda)$ for any $X\in T(\lambda)$.
		\item $\charge(X)\,(\MD d_\lambda)$ for $X\in T(\lambda)$ is equi-distrubuted, i.e.
		$$\#\{X\in T(\lambda)\mid \charge(X)\equiv r\;(\MD d_\lambda)\}$$ is independent of $r\in\left[\overline{d_\lambda}\right]$.
	\end{enumerate} 
\end{Lemma}
\begin{proof}
	It suffices to prove the first result. Let $\lambda=(a_1^{b_1},a_2^{b_2},\ldots)$ where $a_1>a_2>\ldots$. We assume all $b_i>1$, otherwise $d_\lambda=1$ and the claim is trivial. Suppose the row in $X$ containing $\overline{n}$ has length $a_i$. If $\overline{n}$ does not lie in row $b_1+\ldots+b_j$, then $\charge(\phi(X))= \charge(X)-1$; otherwise $\charge(\phi(X))= \charge(X)-1+b_1+\ldots+b_j$. Hence the claim follows from $d_\lambda\mid b_1+\ldots+b_j$.
\end{proof}

For the rest of the section, we restrict ourselves to $\lambda=(l^m)$. From Theorem \ref{knuthQ} and Theorem \ref{knuthrho}, we know there is a bijection:
$$\LKC_{w_0^\lambda}\rightarrow\left\{\left.X\in T(\lambda)\mid m \right|\charge(X)\right\}$$
by $w\mapsto P(w)$. Hence from Lemma \ref{equidist}, we know $\#\LKC_{w_0^\lambda}=\#T(\lambda)/m=\frac{n!}{m(l!)^m}.$ In order to understand the Knuth class containing $w_0^\lambda$ better, we need to construct explicitly the inverse of the above bijection.

We define the following set of affine permutations and study its combinatorial properties:
\begin{Definition}\label{def:fundbox}
Define the \textit{fundamental box} $\mfF$ to be the set of affine permutations $w$ satisfying the following three sets of conditions:
    \begin{enumerate}
	\item $w^{-1}(i)<w^{-1}(i+m)<\ldots<w^{-1}(i+m(l-1))<w^{-1}(i)+n$ for $i\in[m]$;
	\item $w^{-1}(1+(j-1)m)<w^{-1}(2+(j-1)m)<\ldots<w^{-1}(jm)$ for $j\in[l]$;
	\item $\lch_i(P(w^{-1}w_0^\lambda))=\Diff_i(w^{-1})$ for all $i\in[m-1]$ where \begin{equation}\label{eq-diff}
	\Diff_i(w^{-1})=\sum_{j=0}^{l-1}\left(\left\lceil{\frac{w^{-1}(i+1+jm)}{n}}\right\rceil-\left\lceil{\frac{w^{-1}(i+jm)}{n}}\right\rceil\right).
	\end{equation}
\end{enumerate}
\end{Definition}
The name fundamental box comes from the equivalent definition using the notion of alcoves, which we explain in Section 5. We can strengthen the second condition in Definition \ref{def:fundbox} to the following:

\begin{Lemma}\label{lemn-l}
	Given $w\in\mfF^{-1}$, then we have  $$w(i+1+jm)-w(i+jm)\le n-l$$ 
	for $i\in[m-1],j\in[0,l-1].$
\end{Lemma}
\begin{proof}
	Suppose on the contrary that $w(i+1+jm)-w(i+jm)\ge n-l+1$ for some $i\in[m-1],j\in[0,l-1].$ Let $a_x=P_{i,x}(ww_0^\lambda), b_x=P_{i+1,x}(ww_0^\lambda)$ for $x\in\mbZ$. Then $w(i+jm)=a_r$ and $w(i+1+jm)=b_{r+\gamma}$ for some $r\in\mbZ$ and $\gamma=\lch_i(P(ww_0^\lambda))$.
	
	Since $b_{r+\gamma}-a_r\ge n-l+1$, we have $a_r\le b_{r+\gamma}-n+l-1=b_{r+\gamma-l}+l-1 \le b_{r+\gamma-1}$. Let $k$ be the largest index with $a_k< b_{r+\gamma}$. Therefore $$a_r+n-l+2\le a_{k+1}<\ldots< a_{r-1+l}\le a_r+n-1.$$
	
	Claim: $b_{k+\gamma}\ge a_r+n-1$. If $b_{k+\gamma}\le a_r+n-2$, then $$a_r+n-l+1\le b_{r+\gamma}<\ldots<b_{k+\gamma}\le a_r+n-2.$$
	So we get $l$ different integers $b_{r+\gamma},\ldots,b_{k+\gamma},a_{k+1},\ldots, a_{r-1+l}$ in a size $l-1$ interval $[a_r+n-l+1,a_r+n-1]$, which is a contradiction, hence the claim is proved.
	
	Now we have a new matching $a_x<b_{x+\gamma-1}$ for all $x\in \mbZ$, contrary to the definition of local charge and property (3).
\end{proof}
\begin{Proposition}
	\label{Abij}
	The map $$\mfF^{-1} \rightarrow\left\{\left.X\in T(\lambda)\mid m \right|\charge(X)\right\}$$ via $w\mapsto P(ww_0^\lambda)$ is a bijection.
\end{Proposition}
\begin{proof}
	For any $w\in \mfF^{-1}$, Lemma \ref{lemAMBC} tells $P(ww_0^\lambda)_i=\left\{\overline{w(i)},\overline{w(i+m)},\ldots,\overline{w(i+m(l-1))}\right\}$ for $i\in[m]$. By the definition of $\Diff_i$ and $\sum\limits_{i=1}^n w(i)=\frac{n(n+1)}{2}$ we know $m$ divides $$\Diff_1(w)+2\Diff_2(w)+\ldots+(m-1)\Diff_{m-1}(w).$$
	Therefore $m\mid \charge(P(ww_0^\lambda))$ and the map is well-defined.
	
	We construct its inverse as follows. Pick a tabloid $P\in T(\lambda)$ with $m\mid\charge(P)$, and let $$e=\sum_{i=1}^{m-1}\lch_i(P)-\frac{\charge(P)}{m}.$$
	Define $w$ to be $$w(i+jm)=P_{i,j+1-e+\sum_{\alpha=1}^{i-1}\lch_\alpha(P)}$$
	for $i\in [m],j\in[0,l-1]$. It can be checked explicitly that $w\in \mfF^{-1}$ and the procedure gives the inverse of the map in the statement.
\end{proof}
\begin{Theorem}\label{thmALKC}
$$\mfF^{-1}{\cdot}w_0^\lambda=\LKC_{w_0^\lambda},$$
or equivalently,
$$\mfF=w_0^\lambda{\cdot}\RKC_{w_0^\lambda}.$$
\end{Theorem}
\begin{proof}
	From Proposition \ref{Abij}, both sets $\mfF^{-1}{\cdot}w_0^\lambda$ and $\LKC_{w_0^\lambda}$ have the same cardinality so we will only show $\LKC_{w_0^\lambda}\subset \mfF^{-1}{\cdot}w_0^\lambda$. Since $w_0^\lambda\in \mfF^{-1}{\cdot}w_0^\lambda\cap \LKC_{w_0^\lambda}$, it suffices to show that for any $w\in\mfF^{-1}\cap \LKC_{w_0^\lambda}{\cdot}w_0^\lambda$, and $s_kww_0^\lambda\sim_{\LKC}  ww_0^\lambda$, there is $s_kw\in\mfF^{-1}$. Suppose $\overline{k}$ lies in row $i$ of $P(ww_0^\lambda)$ and $\overline{k+1}$ lies in row $i'$. 
	
	Case 1: $i\ne i'\pm 1.$ Then monotone conditions (1) (2) (in Definition \ref{def:fundbox}) of $\mfF^{-1}$ are be preserved when multiplying $s_k$ on the left. Then $P(ww_0^\lambda)$ and $P(s_kww_0^\lambda)$ can be computed directly from Lemma \ref{lemAMBC} and we have $P(s_kww_0^\lambda)=s_kP(ww_0^\lambda)$. If $k\neq 0$, both $\Diff_i$ and $\lch_i$ will not change so (3) is satisfied as well. In case $k=0$, 
	$$\sum_{j=0}^{l-1}\left\lceil{\frac{s_kw(i+jm)}{n}}\right\rceil=\sum_{j=0}^{l-1}\left\lceil{\frac{w(i+jm)}{n}}\right\rceil+1,$$
	$$\sum_{j=0}^{l-1}\left\lceil{\frac{s_kw(i'+jm)}{n}}\right\rceil=\sum_{j=0}^{l-1}\left\lceil{\frac{w(i'+jm)}{n}}\right\rceil-1.$$
	Also we have
	\begin{align*}
	\lch_{i-1}(P(s_kww_0^\lambda)) & =\lch_{i-1}(P(ww_0^\lambda))+1,\\
	\lch_{i}(P(s_kww_0^\lambda)) & =\lch_{i}(P(ww_0^\lambda))-1,\\
	\lch_{i'-1}(P(s_kww_0^\lambda)) & =\lch_{i'-1}(P(ww_0^\lambda))-1,\\
	\lch_{i'}(P(s_kww_0^\lambda)) & =\lch_{i'}(P(ww_0^\lambda))+1.
	\end{align*} 
	Hence $\Diff_j(s_kw)=\lch_j(P(s_kww_0^\lambda))$.
	
	Case 2: $i=i'\pm 1.$ Let $w(i+am)\equiv k\;(\MD n)$ and $w(i'+bm)\equiv k+1\;(\MD n)$.
	
	Case 2.1: $a\ne b.$ Clearly $s_kw$ still satisfies monotone conditions (1) (2) and $P(s_kww_0^\lambda)=s_kP(ww_0^\lambda)$. When $k\ne 0$, $\Diff$ and $\lch$ do not change so (3) is satisfied as well. So we only need to consider $k=0$. If in addition $i'=i+1$, 
	\begin{align*}
	\lch_{i-1}(P(s_kww_0^\lambda)) & =\lch_{i-1}(P(ww_0^\lambda))+1,\\
	\lch_{i}(P(s_kww_0^\lambda)) & =\lch_{i}(P(ww_0^\lambda))-2,\\
	\lch_{i+1}(P(s_kww_0^\lambda)) & =\lch_{i+1}(P(ww_0^\lambda))+1,
	\end{align*} 
	so (3) is again satisfied. $i'=i-1$ is similar.
	
	Case 2.2: $a=b$. $s_kw$ satisfies (1) trivially.
	
	Case 2.2.1: $i'=i+1.$ By Lemma \ref{lemn-l}, $0<w(i+1+am)-w(i+am)\le n-l$, so we have $w(i+1+am)-w(i+am)=1$, and $ww_0^\lambda((a+1)m-i)-ww_0^\lambda((a+1)m+1-i)=1$. Then by definition of Knuth move, $s_kww_0^\lambda\not\sim_{\LKC} ww_0^\lambda$, so this case cannot happen.
	
	Case 2.2.2: $i=i'+1$. Again by Lemma \ref{lemn-l}, there is $0<w(i+1+am)-w(i+am)\le n-l$, hence we know $l=1$, $a=0$ and $w(i+1)-w(i)=n-1$. Suppose $ww_0^\lambda(n-i)=k+\alpha n$ for some $\alpha\in\mbZ$, then $ww_0^\lambda(n-i+1)=k+1+(\alpha-1) n$, and $ww_0^\lambda(2n-i+1)=k+1+\alpha n$. Because $ww_0^\lambda(1)>ww_0^\lambda(2)>\ldots>ww_0^\lambda(n)$, and $((ww_0^\lambda)^{-1}(k-1+\alpha n),k-1+\alpha n)$ or $((ww_0^\lambda)^{-1}(k+2+\alpha n),k+2+\alpha n)$ is $(n,n)$-translate of one of $\{(i,ww_0^\lambda(i))\}_{i=1}^{n}$, hence cannot lie between row $n-i$ and $2n-i+1$, which contradicts $s_kww_0^\lambda\sim_{\LKC} ww_0^\lambda$.
\end{proof}

\begin{Corollary}\label{corconst}
	Let $w\in\mfF^{-1}$. Then $P(ww_0^\lambda)=\overline{w}(T^\lambda), Q(ww_0^\lambda)=T^\lambda$ and $\tilde{\rho}(ww_0^\lambda)$ is a constant vector with value $$-\sum_{i=1}^{m-1}\lch_i(P(ww_0^\lambda))+\frac{\charge(P(ww_0^\lambda))}{m}=\sum_{j=0}^{l-1}\left\lceil\frac{w(1+jm)}{n}\right\rceil-l.$$
\end{Corollary}
\begin{proof}
	From Theorem \ref{thmALKC} we know $ww_0^\lambda\in\mfF^{-1}{\cdot}w_0^\lambda=\LKC_{w_0^\lambda}$. And by Lemma \ref{lemAMBC}, we have
	$$\tilde{\rho}_i(ww_0^\lambda)=\sum_{j=0}^{l-1}\left\lceil\frac{w(i+jm)}{n}\right\rceil-l-\sum_{j=1}^{i-1}\lch_j(P(ww_0^\lambda)).$$
	Denote $P=P(ww_0^\lambda)$ for simplicity. From the construction in Proposition \ref{Abij}, we have
	$$\tilde{\rho}_i(ww_0^\lambda)=\sum_{j=0}^{l-1}\left\lceil\frac{P_{i,j+1-e+\sum_{j=1}^{i-1}\lch_j(P)}}{n}\right\rceil-l-\sum_{j=1}^{i-1}\lch_j(P)=-e=-\sum_{i=1}^{m-1}\lch_i(P)+\frac{\charge(P)}{m}.$$
\end{proof}
\begin{Corollary}\label{corphim}
	$$\phi^m(\LKC_{w_0^\lambda})=\LKC_{w_0^\lambda}.$$
\end{Corollary}
\begin{proof}
	By definition of $\phi$,
	\begin{align*}
	\phi^m(w)= & \left[w(1+(l-1)m)+m-n,w(2+(l-1)m)+m-n,\ldots,w(n)+m-n,\right.\\
	& \left.w(1)+m,w(2)+m,\ldots,w(m)+m,\ldots,w((l-1)m)+m\right].
	\end{align*}
	And it can be checked easily that $\phi^m(w)\in\LKC_{w_0^\lambda}$ if $w\in\LKC_{w_0^\lambda}$ using the three conditions of $\mfF$ and Theorem \ref{thmALKC}.
\end{proof}

	

We now study the explicit form of the right (resp. left) cells in $C_\lambda$ and relate them to the Knuth classes. In particular, we show that $w_0^\lambda{\cdot}R_{T^\lambda}$ equals the set of all affine permutations satisfying only the first two conditions of the fundamental box $\mfF$ (in Definition \ref{def:fundbox}) in the following proposition.

\begin{Proposition}
	\label{propleftcell}
	The left cell $L_{T^\lambda}$ containing $w_0^\lambda$ is the collection of all affine permutations satisfying:
	\begin{enumerate}
		\item $w(i)<w(i+m)<\ldots<w(i+m(l-1))<w(i)+n$ for $i\in[m]$;
		\item $w(1+(j-1)m)>w(2+(j-1)m)>\ldots>w(jm)$ for $j\in[l]$.
	\end{enumerate}
\end{Proposition}
\begin{proof}
	For simplicity, we denote the collection satisfying the two conditions above as $\mcC$.
	
	By affine matrix ball construction, we know the affine permutations in $\mcC$ has recording tabloid $T^\lambda$, hence lying in $L_{T^\lambda}$.
	
	And by \cite{shi2006kazhdan}, any left cell is left-connected, hence it suffices to show that if $w\in\mcC$ and $s_kw\in L_{T^\lambda}$, then $s_kw\in\mcC$. Suppose on the contrary that $s_kw\notin\mcC$, and suppose $w(i+jm)\equiv k \,(\MD n),w(i'+j'm)\equiv k+1 \,(\MD n)$ for some $i,i'\in[m]$ and $j,j'\in[0,l-1]$.
	
	Case 1: $j=j'$. Then if in addition $i<i'$, or $i>i',w(i'+j'm)-w(i+jm)\ge n+1$, we have $s_kw\in\mcC.$ It remains to consider the case $j=j',i=i'+1,$ and $w(i'+j'm)-w(i+jm)=1.$ But in this situation the matrix balls $(i-1+jm,w(i+jm)),(i+jm,w(i+jm)+1)$ can be contained in a stream with density $l+1$, so $\lambda_1(s_kw)\ge l+1,$ which contradicts $s_kw\in L_{T^\lambda}.$
	
	Case 2: $j\ne j'$ and $i\ne i'$, then clearly $s_kw\in\mcC.$
	
	Case 3: $j< j'$ and $i= i'$. Since
	$$w(i)<w(i+m)<\ldots<w(i+(l-1)m)<w(i)+n,$$
	we know $j'=j+1$ and $w(i+(j+1)m)-w(i+jm)=1.$ In this situation, we have an anti-stream of $m+1$ matrix balls in $s_kw$:
	\begin{align*}
	& (1+jm,w(1+jm))\ge_{SW}\ldots\ge_{SW}(i+jm,w(i+jm)+1)\ge_{SW}(i+(j+1)m,w(i+jm))\\
	& \ge_{SW}(i+1+(j+1)m,w(i+1+(j+1)m))\ge_{SW}\ldots\ge_{SW}((j+2)m,w((j+2)m).
	\end{align*}
	Therefore $\lambda^{\Tr}_1(s_kw)\ge m+1,$ which is also a contradiction.
	
	Case 4: $j> j'$ and $i= i'$. This is the same as Case 3 after applying the rotation map $\phi$.
\end{proof}

\begin{Proposition}\label{proprcellstr}
	$$\left\{\left.\phi^k(w{\cdot}R_{T^{\lambda}})=R_{\phi^k(\overline{w}(T^\lambda))}\,\right| k\in[0,m-1],w\in \mfF^{-1}\right\}$$
	is the collection of all right cells in $C_\lambda$. Moreover, $$\ell(ww_0^\lambda w')=\ell(w)+\ell(w_0^\lambda)+\ell( w')$$ for $w\in\mfF^{-1}$ and $w_0^\lambda w'\in R_{T^\lambda}.$ 
\end{Proposition}
\begin{proof}
	For any $w\in\mfF^{-1}= \LKC_{w_0^\lambda}{\cdot}w_0^\lambda$, $w_0^\lambda w'\in R_{T^\lambda}$, we claim:
	\begin{enumerate}
		\item $ww_0^{\lambda}$ and $ww_0^\lambda w'$ are in the same right cell;
		\item $w_0^{\lambda}w'\sim_{\LKC}ww_0^\lambda w'$.
	\end{enumerate}
Since any right cell is right-connected \cite{shi2006kazhdan}, we could find a path (right multiplication by simple reflections) in $R_{T^\lambda}$ connecting $w_0^\lambda$ and $w_0^\lambda w'$. Similarly, we have a path (left Knuth moves) in $\LKC_{w_0^\lambda}$ connecting $w_0^\lambda$ and $ww_0^\lambda$. Hence by induction, it suffices to assume $ww_0^\lambda w',ww_0^\lambda w's_i$ being in the same right cell and $ww_0^\lambda w'\sim_{\LKC}s_jww_0^\lambda w'$, and prove $s_jww_0^\lambda w',s_jww_0^\lambda w's_i$ being in the same right cell and $ww_0^\lambda w's_i \sim_{\LKC}s_jww_0^\lambda w's_i$.
\begin{equation*}
\begin{tikzcd}
ww_0^\lambda w' \arrow[r, "\cdot s_i"'] \arrow[d, "s_j \cdot "'] & ww_0^\lambda w's_i \arrow[d, "s_j\cdot"']\\
s_jww_0^\lambda w' \arrow[r, "\cdot s_i"'] & s_jww_0^\lambda w's_i
\end{tikzcd}
\end{equation*}
Denote $P=P(ww_0^\lambda w')$ and $Q=Q(ww_0^\lambda w')$. Since $ww_0^\lambda w'\sim_{\LKC}s_jww_0^\lambda w'$, we know $P(s_jww_0^\lambda w')=P'$, and $Q(s_jww_0^\lambda w')=Q$ where $P$ and $P'$ differs by a Knuth move. Since $ww_0^\lambda w',ww_0^\lambda w's_i$ are in the same right cell, we have $P(ww_0^\lambda w's_i)=P$, and denote $Q(ww_0^\lambda w' s_i)=Q'$. By \cite[Proposition 3.23]{chmutov2017monodromy}, we know there is a unique affine permutation with insertion tabloid $P'$ and recording tabloid $Q'$ that is related to $ww_0^\lambda w' s_i$ by a left Knuth move. Suppose this affine permutation is $s_{j'}ww_0^\lambda w's_i$. And we claim $j=j'$.

Without loss of generality, suppose for some $\overline{\alpha}\in\left[\overline{n}\right]$, $\overline{\alpha}\in\tau(P)$, $\overline{\alpha+1}\notin\tau(P)$, $\overline{\alpha+1}\in\tau(P')$, $\overline{\alpha}\notin\tau(P')$. Hence by Proposition \ref{propdecent}, $(ww_0^\lambda w')^{-1}(\alpha+1)$ is smallest one among $(ww_0^\lambda w')^{-1}(\alpha)$, $(ww_0^\lambda w')^{-1}(\alpha+1)$, and $(ww_0^\lambda w')^{-1}(\alpha+2)$. Similarly, $(ww_0^\lambda w')^{-1}(\alpha+1)$ is smallest one among $(ww_0^\lambda w's_i)^{-1}(\alpha)$, $(ww_0^\lambda w's_i)^{-1}(\alpha+1)$, and $(ww_0^\lambda w's_i)^{-1}(\alpha+2)$. The only possibilities that $j\ne j'$ are the following two cases:

Case 1:  $(ww_0^\lambda w')^{-1}(\alpha+1)<(ww_0^\lambda w')^{-1}(\alpha)<(ww_0^\lambda w')^{-1}(\alpha+2) $, but $(ww_0^\lambda w's_i)^{-1}(\alpha+1)<(ww_0^\lambda w's_i)^{-1}(\alpha+2)<(ww_0^\lambda w's_i)^{-1}(\alpha) $, so $j=\alpha+1$ and $j'=\alpha$. This happens iff $$i=(ww_0^\lambda w')^{-1}(\alpha)=(ww_0^\lambda w')^{-1}(\alpha+2)-1.$$

Since $ww_0^\lambda w'\in C_{l^m}$, we know each matrix ball in $\mathcal{B}_{ww_0^\lambda w'}$, and in particular $((ww_0^\lambda w')^{-1}(\alpha+2),\alpha+2)=(i+1,\alpha+2)$, is contained in an anti-stream of density $m$. But if we replace $(i+1,\alpha+2)$ with $(i+1,\alpha)$ and $(i,\alpha+2)$, we obtain an anti-stream of density $m+1$ in $\mathcal{B}_{ww_0^\lambda w's_i}$, which contradicts $ww_0^\lambda w's_i$ and $ww_0^\lambda w'$ being in the same right cell.

Case 2: $(ww_0^\lambda w')^{-1}(\alpha+1)<(ww_0^\lambda w')^{-1}(\alpha+2)<(ww_0^\lambda w')^{-1}(\alpha) $, but $(ww_0^\lambda w's_i)^{-1}(\alpha+1)<(ww_0^\lambda w's_i)^{-1}(\alpha)<(ww_0^\lambda w's_i)^{-1}(\alpha+2) $, so $j=\alpha$ and $j'=\alpha+1$. This happens iff $$i=(ww_0^\lambda w')^{-1}(\alpha+2)=(ww_0^\lambda w')^{-1}(\alpha)-1.$$
Similar to the previous case, since $ww_0^\lambda w'\in C_{l^m}$, we know each matrix ball in $\mathcal{B}_{ww_0^\lambda w'}$, and in particular $((ww_0^\lambda w')^{-1}(\alpha),\alpha)=(i+1,\alpha)$, is contained in a stream of density $l$. But if we replace $(i+1,\alpha)$ with $(i,\alpha)$ and $(i+1,\alpha+2)$, we obtain a stream of density $l+1$ in $\mathcal{B}_{ww_0^\lambda w's_i}$, which contradicts $ww_0^\lambda w's_i$ and $ww_0^\lambda w'$ being in the same right cell. 

So both cases cannot happen and we arrived at the claim $j=j'.$

By Proposition \ref{propdecent} and the second length formula in Lemma \ref{lengthformula}, we know that $$\ell(s_jww_0^\lambda w' )-\ell(ww_0^\lambda w' )=\ell(s_jww_0^\lambda w' s_i)-\ell(ww_0^\lambda w' s_i).$$
Therefore by induction there is:
\begin{align*}
\ell(ww_0^\lambda w')=  \,\ell(ww_0^\lambda )+\ell(w_0^\lambda w')-\ell(w_0^\lambda)
=\ell(w)+\ell(w_0^\lambda)+\ell(w')
\end{align*}
and the last equality is due to the monotone condition in Proposition \ref{propleftcell} and the first length formula in Lemma \ref{lengthformula}.

Now from the claim, we know $w{\cdot}R_{T^\lambda}$ is contained in the right cell $R_{\overline{w}(T^\lambda)}$. Moreover, we have a map $R_{T^\lambda}\rightarrow R_{\overline{w}(T^\lambda)}$ by $x\mapsto wx$, and in fact left multiplication by $w^{-1}$ gives an inverse of this map, hence $w{\cdot}R_{T^\lambda}=R_{\overline{w}(T^\lambda)}$. Applying rotations we get $\phi^k(w{\cdot}R_{T^\lambda})=R_{\phi^k(\overline{w}(T^\lambda))}$. Lemma \ref{equidist} indicates these are all the right cells in $C_\lambda.$
\end{proof}
 
For $k\in[m-1]$, let $w^{(k)}$ be the following affine permutation:
\begin{align*}
w^{(k)}= \phi^k\left(\left[ \right.\right. &  m-k+1,m-k+2,\ldots,m,1,2,\ldots,m-k,\\
& 2m-k+1,2m-k+2,\ldots,2m,1+m,2+m,\ldots,2m-k,\\
& \ldots\\
& \left.\left.n-k+1,n-k+2,\ldots,n,1+(l-1)m,2+(l-1)m,\ldots,n-k\right]\right).
\end{align*}
\begin{Lemma}\label{lem:wk}
	$(w^{(k)})^{-1}\in L_{T^\lambda}{\cdot}w_0^\lambda$, but $(w^{(k)})^{-1}\notin\mfF^{-1}$. And there exists some $w\in\mfF^{-1}$ and some $i\in[0,n-1]$, such that $(w^{(k)})^{-1}=s_iw$.
\end{Lemma}
\begin{proof}
	By definition,
	\begin{align*}
	(w^{(k)})^{-1}=\left[\right. & 1-m+k,2-m+k,\ldots,2k-m,2k+1,2k+2,\ldots,m+k,\\
	& 1+k,2+k,\ldots,2k,2k+1+m,2k+2+m,2m+k,\\
	& \ldots\\
	& 1+k+(l-2)m,2+k+(l-2)m,\ldots,2k+(l-2)m,\\
	& \quad\quad\quad\quad\quad\quad\left. 2k+1+(l-1)m,2k+2+(l-1)m,\ldots,ml+k\right]
	\end{align*}
	It can be checked directly that $(w^{(k)})^{-1}\in L_{T^\lambda}{\cdot}w_0^\lambda$, $\Diff_i((w^{(k)})^{-1})=\lch_i(P((w^{(k)})^{-1}w_0^\lambda))$ when $i\ne k$, but $\Diff_k((w^{(k)})^{-1})=\lch_k(P((w^{(k)})^{-1}w_0^\lambda))+1$. So $(w^{(k)})^{-1}\notin\mfF^{-1}$.
	There are different ways to find a pair of required $w$ and $i$. One way is to take
	\[   i=\left\{
	\begin{array}{ll}
	2k, & m>2k \\
	m, & m=2k \\
	2k-m, & m<2k \\
	\end{array} 
	\right. .\]
	And one can check directly that $s_i(w^{(k)})^{-1}\in \mfF^{-1}.$
\end{proof}

\begin{Proposition}\label{Prop:Fi}
	For any $w\in L_{T^\lambda}{\cdot}w_0^\lambda$, either $w\in \mfF^{-1}$, or there exists some $k\in[m-1]$ and $w'\in L_{T^\lambda}{\cdot}w_0^\lambda$, such that $w=\phi^k(w')(w^{(k)})^{-1}$
	and $\Diff(w')-\lch(P(w'w_0^\lambda))<\Diff(w')-\lch(P(w'w_0^\lambda))$. Moreover, $\ell(w)=\ell(w')+\ell(w^{(k)}).$
\end{Proposition}
\begin{proof}
	For any given $w\in L_{T^\lambda}{\cdot}w_0^\lambda\setminus \mfF^{-1}$, there exists some $k\in[m-1]$ such that $\Diff_k(w)>\lch_k(P(ww_0^\lambda))$. Let $w'=\phi^{-k}(ww^{(k)})$. Direct computation gives that $w'$ equals
	\begin{align*}
	\left[\right. & w(m{+}1){-}k,w(m{+}2){-}k,\ldots,w(m{+}k){-}k,w(1{+}k){-}k,w(2{+}k){-}k,\ldots,w(m){-}k,\\
	& w(2m{+}1){-}k,w(2m{+}2){-}k,\ldots,w(2m{+}k){-}k,w(1{+}k{+}m){-}k,w(2{+}k{+}m){-}k,\ldots,w(2m){-}k,\\
	& \ldots\\
	& w(1){+}n{-}k,w(2){+}n{-}k,\ldots,w(k){+}n{-}k,\\
	&\;\;\;\qquad\quad\qquad\qquad\qquad\qquad\qquad\quad \left.w(1{+}k{+}(l{-}1)m){-}k,w(2{+}k{+}(l{-}1)m){-}k,\ldots,w(n){-}k\right].\\
	\end{align*}
	Since $\Diff_k(w)>\lch_k(P(ww_0^\lambda))$, we have $w((j+1)m+k)-k>w(1+k+jm)-k$ for all $j\in[0,l-1]$, and clearly $w'$ preserves monotone conditions from $w$. Moreover,
	\[   \Diff_i(w')-\lch_i(P(w'w_0^\lambda))=\left\{
	\begin{array}{ll}
	\Diff_i(w)-\lch_i(P(ww_0^\lambda)), & i\ne k; \\
	\Diff_k(w)-\lch_k(P(ww_0^\lambda))-1, & i=k. \\
	\end{array} 
	\right. \]
	By Shi's length formula in Lemma \ref{lengthformula}, $$\ell(w)=\ell(\phi^{-k}(w))=\ell(w')+l(m-k)k=\ell(w')+\ell(w^{(k)}).$$
\end{proof}
An immediate corollary is the following:
\begin{Corollary}\label{Cor:Fi}
	For any $w\in w_0^\lambda{\cdot}R_{T^\lambda}$, we have the following expression:
	$$w=w^{(k_1)}\phi^{k_1}(w^{(k_2)}\phi^{k_2}(\cdots w^{(k_\varepsilon)}\phi^{k_\varepsilon}(w')\cdots))$$
	where $w'\in \mfF$ and $\{k_1,\ldots,k_\varepsilon\}=\{1^{d_1},2^{d_2},\ldots,(m-1)^{d_{m-1}}\}$ as a multi-set (the order of $k_i$'s does not matter) and $d_j=\Diff_j(w^{-1})-\lch_j(P(w^{-1}w_0^\lambda))$ for $j\in[m-1]$.
\end{Corollary}

\begin{Proposition}\label{propnumcomb}
	Let $w\in w_0^\lambda{\cdot}R_{T^\lambda}$ and $$\{j_1<j_2<\ldots<j_s\}=\left\{j\in[m-1]\mid \Diff_j(w^{-1})>\lch_j(w^{-1}w_0^\lambda)\right\}.$$
	Then the number of elements in $\Psi^{-1}(w_0^\lambda w)$ is $$\#\Psi^{-1}(w_0^\lambda w)=\#\Psi^{-1}(w^{-1}w_0^\lambda )=\begin{pmatrix}\displaystyle m \\ j_1,j_2-j_1,\ldots,j_s-j_{s-1},m-j_s \end{pmatrix}.$$
\end{Proposition}
\begin{proof}
	Since$$\rho_i(w^{-1}w_0^\lambda)=\sum_{j=0}^{l-1}\left\lceil\frac{w^{-1}(i+jm)}{n}\right\rceil-l,$$we have
	\begin{align*}
	\tilde{\rho}_{i+1}(w^{-1}w_0^\lambda){-}\tilde{\rho}_{i}(w^{-1}w_0^\lambda)=& \sum_{j=0}^{l-1}\left\lceil\frac{w^{-1}(i+1+jm)}{n}\right\rceil{-}\sum_{j=0}^{l-1}\left\lceil\frac{w^{-1}(i+jm)}{n}\right\rceil{-}\lch_i(P(w^{-1}w_0^\lambda))\\
	= & \Diff_i(w^{-1}w_0^\lambda)-\lch_i(P(w^{-1}w_0^\lambda)).
	\end{align*}
	Therefore,
	\begin{align*}
	\tilde{\rho}_1(w^{-1}w_0^\lambda)=\ldots=\tilde{\rho}_{j_1}(w^{-1}w_0^\lambda)<\tilde{\rho}_{j_1+1}(w^{-1}w_0^\lambda)=\ldots=\tilde{\rho}_{j_2}(w^{-1}w_0^\lambda)<\tilde{\rho}_{j_2+1}(w^{-1}w_0^\lambda)=\ldots\\
	\ldots =\tilde{\rho}_{j_s}(w^{-1}w_0^\lambda)<\tilde{\rho}_{j_s+1}(w^{-1}w_0^\lambda)=\ldots=\tilde{\rho}_m(w^{-1}w_0^\lambda).
	\end{align*}
	Hence by Theorem \ref{thminverseAMBC}, $$\#\Psi^{-1}(w_0^\lambda w)=\#\Psi^{-1}(w^{-1}w_0^\lambda )=\begin{pmatrix}\displaystyle m \\ j_1,j_2-j_1,\ldots,j_s-j_{s-1},m-j_s \end{pmatrix}.$$
\end{proof}
\section{Affine Springer fibers and their geometry}

In this section we will introduce the geometric spaces that will appear in this paper, and explain some of their basic properties.

Denote by $\mathcal{K}:=\mathbb{C}((t))$ and $\mathcal{O}:=\mathbb{C}[[t]]$ the Laurent power series and power series algebras with complex coefficients respectively.

Let $G$ be a reductive group, $B\subset G$ a \textit{Borel subgroup}  and $T\subset B$ a \textit{maximal torus}. Denote by $\mathfrak{g}$ the Lie algebra of $G$. The \textit{root system} of $G$ of roots, weights, coroots and coweights is given by $(R,\mathbb{X}, R^\vee,\mathbb{X}^\vee)$. Further the choice of $B$, gives a choice of positive roots $R^+$. We also have the set of \textit{affine roots}, given by $R_{\aff}\coloneqq R\times\mathbb{Z}\delta\cup\{0\}\times\mathbb{Z}\delta$, where $\delta$ is the generator of the affine direction. Associated to this, we have the \textit{Weyl group} $W$, the \textit{affine Weyl group} $\widetilde{W}\coloneqq W\ltimes \mathbb{Z}R^\vee$ and the \textit{extended affine Weyl group} $\widetilde{W}_{\text{ext}}\coloneqq W\ltimes\mathbb{X}^\vee$. In the case of a simply connected group, we have $\mathbb{Z}R^\vee=\mathbb{X}^\vee$ and thus the affine Weyl group and the extended affine Weyl group agree.

We can now construct an \textit{Iwahori subgroup} $I$ of $G(\mathcal{O})\subset G(\mathcal{K})$ via the pullback diagram
\[\begin{tikzcd}
I\ar[r,hookrightarrow]\ar[d] & G(\mathcal{O})\ar[d,"t=0"]\\
B\ar[r,hookrightarrow] & G
\end{tikzcd}\]
With this, we are ready to define the \textit{affine flag variety} $\mcF l$ as the ind-scheme whose closed points are given by the quotient $G(\mathcal{K})/I$. We will only consider the properties of the reduced structure of this space so we omit the details of its scheme structure.

The affine flag variety has a \textit{Schubert decomposition} into locally closed subsets given by $I$-orbits labeled by $\widetilde{W}_{\text{ext}}$. This is given by a natural inclusion $\widetilde{W}_{\text{ext}}=N_{G(\mathcal{K})}(T(\mathcal{K}))/T(\mathcal{O})\hookrightarrow\mcF l$. Considering the $I$-orbits we get
$$\bigsqcup_{w\in\widetilde{W}_{\text{ext}}} IwI/I.$$

This decomposition can also be understood as a decomposition of $\mcF l\times\mcF l$ into $G(\mathcal{K})$-orbits. These again are labeled by $\widetilde{W}_{\text{ext}}$, by considering the $G(\mathcal{K})$-orbit of $(1,w)$. For two points $x,y\in\mcF l$, we say they are in \textit{relative position} $r(x,y)=w$, if $(x,y)\in\mcF l\times\mcF l$ is in the $G(\mathcal{K})$-orbit labeled by $w$. Note that we have $r(x,y)=r(y,x)^{-1}$ and $r(gx,gy)=r(x,y)$ for any $g\in G(\mathcal{K})$.

This decomposition also induces a partial order into $\widetilde{W}_{\text{ext}}$, known as the \textit{Bruhat order}, given by $w\leq w'$ if $IwI/I\subset\overline{Iw'I/I}$ or equivalently $G(\mathcal{K})(1,w)\subset \overline{G(\mathcal{K})(1,w')}$.

We can now extend the notion of relative position to pairs of irreducible subvarieties of $\mcF l$. Namely let $X,Y\subset\mcF l$ be two irreducible subvarieties. Then $X\times Y\subset\mcF l\times\mcF l$ has a stratification into locally closed subsets given by intersections with $G(\mathcal{K})$-orbits. As $X\times Y$ is irreducible there is a unique $G(\mathcal{K})$-orbit, say labeled by $w$, that intersects $X\times Y$ in an open subset. We then denote $r(X,Y)=w$ and say $X$ and $Y$ are in relative position $w$. Note that as a consequence $X\times Y\subset \overline{G(\mathcal{K})(1,w)}$. Thus for any pair of points $x\in X$, $y\in Y$, we have $r(x,y)\leq r(X,Y)$ and generically this is an equality. The same properties as above are thus easy to see, i.e. $r(X,Y)=r(Y,X)^{-1}$ and $r(gX,gY)=r(X,Y)$ for any $g\in G(\mathcal{K})$.

We also have a similar definition of relative position for the \textit{finite flag variety} $G/B$ in terms of $B$-orbits of $G/B$ and $G$-orbits of $G/B\times G/B$ given by elements of the Weyl group $W$, but we omit the details as it is essentially the same construction as above.

We can now introduce the \textit{affine Springer fiber} associated to an element $\gamma\in\mathfrak{g}(\mathcal{K})$ following the work of \cite{KL}. This is a subscheme $\mcF l_\gamma\subset\mcF l$ with closed points given by
$$\mcF l_\gamma\coloneqq\{gI\in\mcF l\mid\gamma\in{}^g\text{Lie}(I)\}.$$

The space $\mcF l$ has automorphisms given by left multiplication with elements of $G(\mathcal{K})$. These automorphisms preserve $\mcF l_\gamma$ if they centralize $\gamma$.

We will now focus on the $G=SL_n$. Recall that for $SL_n$, $\mathbb{X}^\vee\coloneqq\{\mu\in\mathbb{Z}^n\mid \sum_i \mu_i=0\}$. The roots are given by $R=\{e_i-e_j\mid i,j\in[n]\}$ and $R^+=\{e_i-e_j\mid i\leq j\}$. The Weyl group for $SL_n$ is $W=S_n$ and the affine Weyl group is $\widetilde{W}=\widetilde{S_n}$. Since $SL_n$ is simply connected, the extended affine Weyl group is the same as the affine Weyl group as stated above.

We consider the affine Springer fiber for $SL_n$ and $\gamma=N\in\mathfrak{sl}_n(\mathcal{O})$ a generic lift of a nilpotent of $\mathfrak{sl}_n$ corresponding to the partition $(l^m)$. This is the nil-element considered in Lusztig's conjecture \cite{LusztigConjclass} in the case of the nilpotent corresponding to the partition $(l^m)$. We can consider the explicit choice given by 
$$N=\left[\begin{array}{c|c|c|c}
0& \text{I} & \ldots & 0\\
\hline
\vdots & &\ddots & 0\\
\hline
0&0&\ldots& \text{I}\\
\hline
th&0&\ldots& 0
\end{array}\right]
.$$
Here the blocks are $m\times m$ matrices, $\text{I}$ is the identity matrix and $h$ is a regular semisimple element, which without loss of generality can be taken to be a diagonal matrix with distinct non-zero eigenvalues.

Note that $N$ is conjugate to the following block diagonal matrix
$$\begin{bmatrix}
J_1 & 0 &\ldots & 0\\
0 & J_2 &\ldots &0\\
\vdots &&\ddots & \vdots\\
0 &0&\ldots & J_m
\end{bmatrix}$$
where the $l\times l$ diagonal block matrices are $$J_i\coloneqq\begin{bmatrix}0&1&0&\ldots & 0\\ 0&0&1&\ldots &0\\ \vdots&&&\ddots &\vdots\\
0&0&0&\ldots&1\\
th_i&0&0&\ldots &0\end{bmatrix}.$$ 

This matrix is centralized by the block diagonal matrices $f_i$, where all the blocks are replaced with the identity except $J_i$. After conjugating this gives matrices $f_i'$ centralizing $N$. Note that these matrices are elements of $GL_n(\mathcal{K})$, but not $SL_n(\mathcal{K})$, in fact their determinant has valuation $1$. We now introduce a more explicit description of $\mcF l$ for $SL_n$ and then see how these matrices give automorphisms of $\mcF l_N$.

To do this we introduce the notion of a $\mathcal{O}$-lattice $V$ inside $\mathcal{K}^n$, as a $\mathcal{O}$-submodule, such that $V\cong \mathcal{O}^n$ as an $\mathcal{O}$-module. Note that $\bigwedge^n\mathcal{K}^n=\mathcal{K}$, and thus $\bigwedge^n V\subset\mathcal{K}$ is a rank $1$ free $\mathcal{O}$-module and thus we must have $\bigwedge^n V=t^k\mathcal{O}\subset\mathcal{K}$ for some $k$.

With this description we have that for $SL_n$:

$$\mcF l=\{V_0\subset V_1\subset\ldots V_{n-1}\subset t^{-1}V_0\mid V_i\text{ is a lattice in }\mathcal{K}^n\text{ and }\bigwedge^n V_i=t^{-i}\mathcal{O}\}.$$

We can now see that $\bigwedge^n g V=\det(g)\bigwedge^n V$ and thus we have the automorphism of $\mcF l$ given by $(F_i(V_j))_k=(f'_i)^{-1} V_{k-1}$. This has the correct determinant and further as $f'_i$ centralize $N$ this induces automorphisms $F_i$ on $\mcF l_N$.\label{def:Fi}

We will consider automorphisms $F^c\coloneqq F_1^{c_1}\ldots F_m^{c_m}$, for $c$ an $m$-tuple. This is well-defined as the $F_i$ clearly commute with each other. Note that if $\sum c_i=0$ we have $(f')^c=(f'_1)^{-c_1}\ldots (f'_m)^{-c_m}$, which is indeed an element of $SL_n(\mathcal{K})$ and so $F^c$ is just multiplication by an element in $SL_n(\mathcal{K})$. We refer to these transformations as the \textit{translations} and the \textit{group of translations} is denoted by $\Lambda$.

We introduce the parahoric $\widetilde{P}$ generated by the Iwahori $I$ and the simple reflections $s_i$ such that $i\not\equiv 0\;(\text{mod }m)$. We will also need the \textit{parabolic subgroup} $W_P\subset\widetilde{W}$ corresponding to this parahoric, i.e. the subgroup generated by the reflections $s_i$ such that $i\not\equiv 0\;(\text{mod }m)$ as above. For this parabolic subgroup $W_P$, the maximal element is $w_P=w_0^\lambda$ as defined after Theorem \ref{cell:thm}.

We can consider the family of spaces $\mcF l_{N}$ over the space $S^{rs}$ of regular semisimple diagonal $m\times m$-matrices $h$ with non-zero eigenvalues, given by varying $N$ in the obvious way. This is a fiber bundle where each fiber is homeomorphic and thus we can consider the monodromy action on components. Note that the monodromy acts on the points of $\widetilde{W}$ through $S_m\hookrightarrow (S_m)^l\cong W_P\subset\widetilde{W}$. We will use this to get an action of $S_m$ on the set of components.

\section{Components of affine Springer fiber of rectangular type}

In this section we study the points of $\widetilde{W}$ appearing in $\mcF l_N$ as well as the intersection with the orbits of the \textit{parahoric} $\widetilde{P}$ introduced in the previous section. To study this we require certain equations that recur throughout the paper. To introduce them, recall the coweight lattice $\mathbb{X}^\vee$ with the action of $\widetilde{W}$. Further consider the set of connected components of $\mathbb{X}^\vee\otimes\mathbb{R}\smallsetminus\cup_{\alpha}\{\langle\alpha,\mu\rangle\in\mathbb{Z}\}$. The closures of these are known as the set of \textit{alcoves} $\mathfrak{A}$ and we have that $\widetilde{W}$ acts on $\mathfrak{A}$. Consider the \textit{fundamental alcove} given by $$A_0:=\{0\leq\langle\alpha,\mu\rangle\leq 1,\forall\alpha\in R^+\}.$$
Then we have a bijection $\widetilde{W}\cong\mathfrak{A}$ given by $w\mapsto A_w\coloneqq w(A_0)$, i.e. by acting on the fundamental alcove $A_0$. 

The following three sets of equations on alcoves are used throughout this section:

\begin{enumerate}
	\item $
    0\leq\langle e_{i+am}-e_{i+bm} ,A_w\rangle\leq 1 \text{ for } i\in[m]\text{ and }0\le a<b\le l-1,$
	\item $
	    \langle e_{i+jm}-e_{i+1+jm} ,A_w\rangle \ge 0\text{ for } i\in[m-1]\text{ and }j\in[0,l-1],
	$
	\item[$(3)_i$] ($i\in[m-1]$) At least one of the following is satisfied:
	\begin{equation*}
	\begin{split}
	& 0\leq\langle e_{i+1}-e_{i+m},A_w\rangle \leq 1,\\
	& 0\leq \langle e_{i+1+m}-e_{i+2m},A_w\rangle \leq 1,\\
	& \ldots\\
	& 0\leq \langle e_{i+1+(l-2)m}-e_{i+(l-1)m},A_w\rangle \leq 1,\\
	& 0\leq\langle e_{i}-e_{i+1+(l-1)m},A_w\rangle \leq 1.
	\end{split}
	\end{equation*}
	We say $A_w$ satisfies equation (3) if it satisfies $(3)_i$ for all $i\in[m-1].$
\end{enumerate}

In fact these three sets of equations give an equivalent definition of the fundamental box by the following proposition.

\begin{Proposition}\label{prop:A}
    $w\in\mfF$ if and only if the corresponding alcove $A_w$ satisfies equations (1), (2) and (3).
\end{Proposition}

\begin{proof}
	Since $$A_w=w(A_0)=\left\{\mu\mid \mu_{w(n)}<\mu_{w(n-1)}<\ldots<\mu_{w(1)}<\mu_{w(n)}+1\right\},$$
	equations (1) and (2) are equivalent to first two conditions in Definition \ref{def:fundbox}, and we show the third are the same as well.
	
Let $w\in \mfF$, and suppose on the contrary that $w$ does not satisfy equation (3). Then from the monotone conditions of $w^{-1}$, there exists some $i\in[m-1]$:
\begin{equation*}
\begin{split}
& w^{-1}(i+1)>w^{-1}(i+m),\\
& w^{-1}(i+1+m)>w^{-1}(i+2m),\\
& \ldots\\
& w^{-1}(i+1+(l-2)m)>w^{-1}(i+(l-1)m),\\
& w^{-1}(i+1+(l-1)m)>w^{-1}(i)+n.\\
\end{split}
\end{equation*}
But these inequalities imply $\lch_i(P(w^{-1}w_0^\lambda))<\Diff_i(w^{-1})$, which contradicts $w\in \mfF$.

Now for any $w$ satisfying all three sets of equations, we know $\Diff_i(w^{-1})\ge\lch_i(P(w^{-1}w_0^\lambda))$ for all $i\in[m-1]$ by (1) and (2). If for some $i$, $\Diff_i(w^{-1})\ge\lch_i(P(w^{-1}w_0^\lambda))+1$, then the $l$ inequalities above holds, which is contradictory to equation $(3)_i$.
\end{proof}

We now begin by understanding the points of $\widetilde{W}$ that are contained in $\mcF l_N$.

\begin{Lemma}
	$w\in \widetilde{W}\cap\mcF l_N$ $\Leftrightarrow$ the alcove $A_w$ satisfies equation $(1)$.
\end{Lemma}

\begin{proof}
	$w\in \widetilde{W}\cap\mcF l_N$ $\Leftrightarrow$ ${}^w N\in \text{Lie}(I)$.
	Note that $N$ has a non-zero entry in the weight spaces $e_i-e_{i+m}$ for $i+m\leq ml$ and $e_{i+(l-1)m}-e_i+\delta$ for $i\leq m$.
	
	Thus from the above we have $w\in\mcF l_N\Leftrightarrow$ $w^{-1}(\alpha)$ is a positive root $\alpha\in R_{\aff}$ for the root spaces with non-zero entries in $N$. These conditions translate to the conditions
	$$0\leq \langle e_i-e_j, A_w\rangle\leq 1$$
	for $i\equiv j\;(\text{mod m})$ and we thus get the elements in $\widetilde{W}\cap\mcF l_N$ are exactly described by equation $(1)$.
\end{proof}

To understand the components better, we consider $\widetilde{P}$-orbits. These turn out to be a disjoint union of smooth open subsets of a number of components. To prove this statement we follow ideas of \cite{GKM}.

Before we start we introduce some torus actions on $\mcF l_N$. First consider $T\subset SL_{n}$ the diagonal torus. We can construct a subtorus 
$$S=\{\text{diag}(\mu_i)\mid\mu_i=\mu_j\text{ if }i\equiv j\,(\text{mod m})\}.$$
Since $S$ commutes with $N$, $S$ acts on $\mcF l_N$.

We also have an action of $\mathbb{G}_m$ on $\mcF l$ by loop rotations, i.e. by scaling the uniformizer $t$ of the algebra $\mathcal{O}$ of power series.

We can then consider the following cocharacter $\mu:\mathbb{G}_m\rightarrow T\times\mathbb{G}_m$ described by $\mu(x)=(\text{diag}(\mu_i(x)),\delta(x))$, such that $\mu_i=x^{-\left\lfloor\frac{i}{m}\right\rfloor}$ and $\delta(h)=x^{-l}$. We can check that this acts by scaling $N$ and thus acts on $\mcF l_N$.

Consider the Lie algebra $\widetilde{\mathfrak{p}}$ of $\widetilde{P}$. Then the action through $\mu$ gives a grading on $\widetilde{\mathfrak{p}}$, with non-negative weights. Denote the graded pieces by $\widetilde{\mathfrak{p}}_k$ and the filtered pieces $\widetilde{\mathfrak{p}}_{>k}=\oplus\widetilde{\mathfrak{p}}_k$. The $0$ graded part is the Levi of $\widetilde{\mathfrak{p}}$, which we denote by $\widetilde{\mathfrak{l}}\coloneqq\widetilde{\mathfrak{p}}_0$ and by $\widetilde{L}$ the corresponding Levi subgroup. Further note that the $S$ acts on $\widetilde{\mathfrak{l}}$ and there is a cocharacter of $S$ such that $\text{Lie}(I)\cap\widetilde{\mathfrak{l}}$ is exactly the non-negative weight spaces.

Before stating the following Lemma for $w$ satisfying equation $(2)$, we introduce the notation $Y_w^\circ=\widetilde{P}wI/I\cap \mcF l_N$ and $Y_w=\overline{\widetilde{P}wI/I\cap \mcF l_N}$. Note that if we consider $w$ satisfying equation $(2)$, we do indeed get all the $\widetilde{P}$-orbits.

\begin{Lemma}\label{fixpoints:P:orbits}
	If $Y_w$ is non-empty, then $W_Pw\cap\mcF l_N$ is non-empty.
\end{Lemma}

\begin{proof}
	Consider the cocharacter $\mu:\mathbb{G}_m\rightarrow T\times\mathbb{G}_m$ defined above.
	
	Note that with respect to this cocharacter $\widetilde{P}$ has all non-negative root spaces. Thus we get $\widetilde{P}$-orbits are contracted by the action through $\lambda$ to the orbit of $\widetilde{L}$ on the points of $\widetilde{W}$.
	
	We thus must have if $\widetilde{P}$-orbit intersects $\mcF l_N$ then the $\widetilde{L}$-orbit on the points of $\widetilde{W}$ also intersects $\mcF l_N$. Note that $S$ acts on $\widetilde{L}$ and so acts on the $\widetilde{L}$-orbits. Then we can consider a cocharacter of $S$ such that the non-negative weights on $\widetilde{\mathfrak{l}}$ are exactly given by the intersection of $\widetilde{\mathfrak{l}}$ with $\text{Lie}(I)$. Under this action the $\widetilde{L}$-orbit on the points of $\widetilde{W}$ contracts to one of those points.
	
	We thus have that if $\widetilde{P}$-orbit intersects $\mcF l_N$, then it intersects it at a point of $\widetilde{W}$ of that orbit. Theses points are exactly as described in the statement of the lemma and thus the result follows.
\end{proof}

\begin{Corollary}
	All the non-empty $Y_w$ are given by $w$ satisfying equations $(1)$ and $(2)$ (i.e. $w\in w_0^\lambda{\cdot}R_{T^\lambda}$).
\end{Corollary}

\begin{proof}
    Note that by the previous lemma $Y_w$ is non-empty if and only if some element in $W_Pw$ satisfies equation $(1)$.
    
    It follows from Lemma \ref{Porbit:combinatorial} that if $w'$ satisfies equation $(1)$ then there is an element in $W_Pw'$ satisfying equations $(1)$ and $(2)$.
\end{proof}

It follows from this Corollary and Proposition \ref{propleftcell} that 
$$\mcF l_N=\bigcup_{w\in w_0^\lambda{\cdot} R_{T^\lambda}} Y_w.$$
 
\begin{Lemma}\label{smooth:open}
	$Y_w^\circ$ is always smooth and equidimensional of dimension $\dim(\widetilde{P}/I)$, when it is non-empty.
\end{Lemma}

\begin{proof}
    The following proof follows \cite{GKM}.
    
	We want to understand the tangent spaces to $Y_w^\circ$ for $w$ satisfying equation $(1)$.
	
	Note that the $\widetilde{P}$-orbit at $w$ is isomorphic to 
	$$\widetilde{P}/\widetilde{P}\cap{}^{w^{-1}}I.$$
	We have the tangent bundle on this space given by $\widetilde{\mathfrak{p}}/\widetilde{\mathfrak{p}}\cap{}^{w^{-1}g^{-1}}\text{Lie}(I)$ at the point $gwI/I$.
	
	Over the intersection with $\mcF l_N$, we have the map $$\text{ad}(N):\widetilde{\mathfrak{p}}/\widetilde{\mathfrak{p}}\cap{}^{w^{-1}g^{-1}}\text{Lie}(I)\rightarrow\widetilde{\mathfrak{p}}/\widetilde{\mathfrak{p}}\cap{}^{w^{-1}g^{-1}}\text{Lie}(I)$$
	given by the adjoint action. Note that the image is in the nilpotent radical $u_{\widetilde{\mathfrak{p}}}$ of $\widetilde{\mathfrak{p}}$. In fact the map
	$$\text{ad}(N):\widetilde{\mathfrak{p}}/\widetilde{\mathfrak{p}}\cap{}^{w^{-1}g^{-1}}\text{Lie}(I)\rightarrow u_{\widetilde{\mathfrak{p}}}/u_{\widetilde{\mathfrak{p}}}\cap{}^{w^{-1}g^{-1}}\text{Lie}(I)$$
	is surjective.
	
	The tangent space at a point in $Y_w^\circ$ is given by the kernel of the above map.
	It follows, as this map is surjective, that the dimension of all tangent spaces is the same and thus the intersection is smooth.
	
	To compute the dimension of each component, we just need to compute the tangent space at any point, which is given by the kernel of the above map of vector bundles. But note that the dimension of this is just the codimension of the above vector spaces, thus the dimension of the tangent space is the codimension of $\widetilde{\mathfrak{p}}/\widetilde{\mathfrak{p}}\cap{}^{w^{-1}}\text{Lie}(I)$ and  $u_{\widetilde{\mathfrak{p}}}/u_{\widetilde{\mathfrak{p}}}\cap{}^{w^{-1}}\text{Lie}(I)$.
	This is exactly given by the dimension of $\widetilde{\mathfrak{l}}/\widetilde{\mathfrak{l}}\cap{}^{w^{-1}}\text{Lie}(I)$, where $\widetilde{\mathfrak{l}}$ is the Lie algebra of the Levi subgroup $\widetilde{L}$.
	
	The intersection $\widetilde{\mathfrak{l}}\cap{}^{w^{-1}}\text{Lie}(I)$ always gives a Borel subalgebra of $\widetilde{\mathfrak{l}}$ and hence the above space is always of dimension $\dim(\widetilde{P}/I)$ as required.
\end{proof}

It follows from Lemma \ref{smooth:open} that the intersection with $\widetilde{P}$-orbits give disjoint smooth open sets of components. To understand the components we want to identify precisely when the intersection is irreducible. We state the exact conditions in the following lemma, but first we introduce some geometric spaces known as \textit{Hessenberg varieties}.

Consider $G$ a reductive group and $B$ a Borel subgroup. Further consider a $G$-representation $V$ and a $B$ stable subrepresentation $W\subset V$ (this notation should not cause any confusion with the notation for the Weyl group). Further consider a vector $v\in V$. Then we can construct the Hessenberg variety
$$H_v^W:=\{gB\in G/B\mid v\in gW\subset V\}.$$

With this we can state the lemma.

\begin{Lemma}\label{connected:p:orbit}
	 For $A_w$ satisfying equations $(1)$ and $(2)$, $Y_w$ is irreducible if and only if $w\in \mathfrak{F}$, i.e. also satisfies equations $(3)$.
\end{Lemma}

\begin{proof}
	Consider the filtration on $\widetilde{P}$ induced by the filtration $\widetilde{\mathfrak{p}}_{>k}$, which we denote by $\widetilde{P}_{>k}$. We can use this filtration to construct the following quotients $$\widetilde{P}_{>k}\backslash\widetilde{P}/\widetilde{P}\cap{}^{w^{-1}}I$$ factoring the map $\widetilde{P}/\widetilde{P}\cap{}^{w^{-1}}I\rightarrow \widetilde{P}/I$.
	
	This induces a similar sequence of maps on the intersection $\widetilde{P}wI/I\cap\mcF l_N$. Each of the sequence of maps induced on this space is an affine space bundle over its image. This follows from the results from \cite{GKM}.
	
	Further from \cite{GKM} we also get explicitly the image on $\widetilde{P}/I$. This image is given by a Hessenberg variety. Namely consider the representation $V=\widetilde{\mathfrak{p}}_1$, which is a $\widetilde{\mathfrak{l}}$ and also a $\widetilde{L}$-representation. Note that $N\in\widetilde{\mathfrak{p}}_1$ and consider $W_w:={}^{w^{-1}}\text{Lie}(I)\cap\widetilde{\mathfrak{p}}_1$ for $w$ satisfying equations $(1)$ and $(2)$. This is stable under the action of the Borel $B_{\widetilde{L}}\subset\widetilde{L}$ given by the image of $I$.
	
    As proven in \cite{GKM}, the image is given by the Hessenberg variety $H^{W_w}_N$.
    Note further it follows from \cite{GKM} or from the above descriptions that $H^{W_w}_N$ is smooth. It thus follows that to prove irreducibility of the $\widetilde{P}$-orbit and hence of $Y_w$ it is equivalent to proving irreducibility of the Hessenberg variety of $H^{W_w}_N$ and hence on the connectedness of this Hessenberg variety.
	
	Using the action by a cocharacter of $S$ on this Hessenberg variety we get that $B_{\widetilde{L}}$-orbits are the attracting sets and that to prove connectedness we just need to prove that all the points in $W_P$ contained in this Hessenberg variety are indeed in the same connected component.
	
	We now give a more explicit description of the flag variety $\widetilde{P}/I$ and of the above Hessenberg variety. Note that the flag variety $\widetilde{P}/I$ is just given by 
	$(SL_m/B)^l$, so it is just given by complete flags of $l$ distinct $m$-dimensional vector spaces. Let us denote them by $V^i$ for $i=1,\ldots, l$.
	
	The representation $V$ of $\widetilde{L}$ can be described as the set $\oplus \Hom(V^i,V^{i+1})$, where we interpret $i+1\;(\text{mod l})$. With this description and an appropriate choice of basis we can describe $N\in V$ as the identity map between $V^i\rightarrow V^{i+1}$ for $i=1,\ldots ,l-1$ and a diagonal map with distinct eigenvalues for the map $V^l\rightarrow V^1$.
	
	With this construction it is clear that we can compose all the maps $\Hom(V^i,V^{i+1})$ to get a map $\Hom(V^k,V^k)$. We thus get several maps from the Hessenberg variety $H^{W_w}_N$ to some Hessenberg variety $H_h^W$ on $SL_m/B$ with the representation of $\text{End}(V^k)$ and the regular semisimple endomorphism $h$, introduced above, as the vector $v$ and some $B$ stable subrepresentation $W$. These maps are given by projecting to the $k$-th factor of $(SL_m/B)^l$. It is easy to see that all the points of the Weyl group $S_m$ are in the image of this map.
	
	These Hessenberg varieties $H_h$ are studied in \cite{boixeda2019fix,Boixedathesis}. It is proven there that for a $B$ stable subrepresentation $\text{Lie}(B)\subset W\subset \text{End}(V^k)$ the Hessenberg variety is connected if and only if all negative finite simple roots are contained in $W$. If $H^{W_w}_N$ is connected, then the image Hessenberg variety $H^W_h$ must be connected as well. The condition for $W$ to satisfy that $H^W_h$ is connected, is equivalent to equation $(3)$ for $w$.
	
	Thus $w\in\mathfrak{F}$ is necessary and it remains to prove it is sufficient.
	
	We thus need to prove that the points in $W_P=(S_m)^l$ appearing in the Hessenberg variety are all on the same component. To do this, we break it up in two steps. We prove in the following Lemma \ref{connected:fibers} that the fibers over $S_m$ of the map to $SL_m/B$ above are connected.
	
	Note that the diagonal $S_m\hookrightarrow (S_m)^l$ is always included in the Hessenberg variety and by Lemma \ref{connected:fibers} we have that every point of $W_P$ in our Hessenberg variety is in the same connected component as one of the diagonal ones. It remains to show that the diagonal ones are in the same connected component. To do this it is enough to proof that $(w)$ and $(ws_i)$ for a simple reflection $s_i$ are in the same component. Without loss of generality we can assume $w=id$.
	
	Consider the $SL_2^{\alpha_i}\hookrightarrow SL_m$ the subgroup corresponding to the simple root $\alpha_i$. Let $g\in SL_2^{\alpha_i}$. We can consider a subspace given by points $(x_i)\in (SL_m/B)^l$ where $x_i=gB$ or $x_i=hgB$. Note that the flags $(G_i)$ corresponding to $gB$ and $hgB$ are given by $G_j=\langle e_1,\ldots, e_j\rangle$ for $j\neq i$. Thus to check whether such a point lies in the space we only need to check the conditions for the $i$th vector space in each flag.
	
	Note that $hG_i\subset G_{i+1}$ and $hG_{i-1}\subset G_i$. Thus the tuple $(x_k)$ satisfies the conditions regardless of the choice $gB$ or $hgB$ for $x_k$, unless one of the Hessenberg conditions is $NG_i^k\subset G_i^{k+1}$. If we have this condition we must have $x_k=x_{k+1}$ in the case $k\neq l$ and $x_l=gB$ $x_1=hgB$ in the case $k=l$.
	
	Equation $(3)_i$ is equivalent to the condition $NG_i^k\subset G_i^{k+1}$ not being required for every $k$. Under those conditions it is easy to see that we can indeed choose $(x_i)$ satisfying all Hessenberg conditions for any $g\in SL_2^{\alpha_i}$. This gives a $\mathbb{P}^1$ inside the Hessenberg variety containing $(id)$ and $(s_i)$. It follows these two are in the same connected component and the result follows.
	\end{proof}
	
\begin{Lemma}\label{connected:fibers}
    The fibers of the map of Hessenberg varieties $H^{W_w}_N\rightarrow H_h$ at $S_m$ are connected, when $w$ satisfies equation $(1)$ and $(2)$.
\end{Lemma}

\begin{proof}
    Without loss of generality we assume we are projecting to the last copy of $SL_m/B$ in $(SL_m/B)^l$.
    
    We give certain conditions under which two elements of $(S_m)^l$ in the fiber of the above map are in the same connected component of the Hessenberg variety $H^{W_w}_N$. We will then use this construction to prove every point of $S_m^l$ appearing in the Hessenberg variety in a fixed fiber is in the same connected component. It follows from the discussion in the proof of Lemma \ref{connected:p:orbit} that this proves the fibers of the map are indeed connected.
    
    Consider the subgroup $SL_2^\alpha\hookrightarrow SL_m$ corresponding to the positive root $\alpha$.
    
    Consider two points $(w_i),(w'_i)\in (S_m)^l\cap H^{W_w}_N$ that satisfy
    $$w'_i=\begin{cases} s_\alpha w_i,& \text{ if }k_1\leq i\leq k_2\\
    w_i,&\text{ otherwise}
    \end{cases}$$
    for some $k_1<k_2$. To see that these two points lie in the same connected component consider the subspace of $(SL_m/B)^l$ given by coordinates $(x_i)$ with
    $$x_i=\begin{cases} g w_iB/B,& \text{ if }k_1\leq i\leq k_2\\
    w_iB/B,&\text{ otherwise}
    \end{cases}$$
    for $g\in SL_2^\alpha$.
    
    We check this is in the Hessenberg variety $H^{W_w}_N$. The only conditions that need to be checked are in the boundry cases $k_i$. Note here that the condition at $k_1$ only depends on the flags of $w_{k_1}$ and $gw_{k_1+1}$ and the condition for $N$ at $i\neq l$ only depends on the relative position, i.e. it only depends on the Schubert cell in which the flag corresponding to $w_{k_1}^{-1}gw_{k_1+1}$ lies.
    
    Note that the action of the torus $S$ makes the set of flags given by $w_{k_1}^{-1}gw_{k_1+1}$ an $S$-stable $\mathbb{P}^1$ with fixed points given by $w_{k_1}^{-1}w_{k_1+1}$ and $w_{k_1}^{-1}s_\alpha w_{k_1+1}$. The Schubert cells are just attracting sets along some cocharacter of $S$. It follows that the relative position of $w_{k_1}^{-1}gw_{k_1+1}$ with the identity is given by either $w_{k_1}^{-1}w_{k_1+1}$ or $w_{k_1}^{-1}s_\alpha w_{k_1+1}$. By assumption both of these do satisfy the conditions of the Hessenberg variety and it thus follows that the above subspace given by $(x_i)$ lie in the Hessenberg variety. This is a connected subvariety containing both $(w_i)$ and $(w'_i)$, hence these two points are in the same connected component.
    
    Now we check that using the above condition, everything can be shown to be in the same connected component. To do this we will consider without loss of generality the fiber at the identity of $S_m$. Then we show every point is in the same connected component as the diagonal identity fixed point.
    
    To do this consider a point $(w_i)\in (S_m)^l\cap H^{W_w}_N$. We prove that if $(w_i)\neq(id)$ there is a point $(w'_i)\in (S_m)^l\cap H^{W_w}_N$ satisfying the conditions above for some positive root $\alpha$ and some $k_1<k_2$ such that the length of $(w'_i)^{-1}w'_{i+1}$ is at most the length of $w_i^{-1}w_{i+1}$ and at least one of them is strictly smaller, where here we consider $i+1$ $(\text{mod } l)$. The result follows by induction as $w_l=id$ by assumption.
    
    Note that after multiplying by $s_\alpha$ we get $(w'_i)^{-1}w_{i+1}'{=}w_i^{-1}w_{i+1}$ or $(w'_i)^{-1}w_{i+1}'{=}w_i^{-1}s_\alpha w_{i+1}$. We will want to check what the condition is for $w_i^{-1}s_\alpha w_{i+1}$ having smaller length than $w_i^{-1} w_{i+1}$. This can be checked to happen exactly when one of $w_i^{-1}(\alpha)$ and $w_{i+1}^{-1}(\alpha)$ is a positive root and the other is negative.
    
    Note that $w_l^{-1}(\alpha)$ is positive, so writing the sign of $w_i^{-1}(\alpha)$ in a string, we have a sequence of $+$ and $-$ starting and ending with $+$. We can then choose a consecutive substring of all $-$'s and multiply those elements by $s_\alpha$ to get $(w'_i)$. By the above description this reduces the length of some products $w_i^{-1}w_{i+1}$ and leaves others unchanged.
    
    The only remaining thing to check is that $(w'_i)$ is indeed in the Hessenberg variety $H^{W_w}_N$. The conditions of the Hessenberg variety are given by conditions on $\Hom(V^k,V^{k+1})$. These conditions only depend on $w_i^{-1}w_{i+1}$. Further the conditions are given by some bounds in Bruhat order of $w_i^{-1}w_{i+1}$, but by construction $(w'_i)^{-1}w_{i+1}'$ is smaller in Bruhat order. It follows that this point also satisfies the conditions of the Hessenberg variety.
    
    To finish note that if there is no $-$'s string in the sequence constructed above for any positive root $\alpha$, then we must have $w_i=id$ $\forall i$ and thus the the connectedness follows.
\end{proof}

We now are ready to give a formula for the relative position of $Y_w$ for $w\in\mathfrak{F}$ to any $Y_{w'}$. To introduce this we recall the notation $w_0^\lambda=w_P\in W_P$ for the longest element in the finite parabolic group $W_P$.

\begin{Theorem}\label{lem:relativepositionforfundbox}
	If $w\in\mathfrak{F}$ and $w'$ satisfies equations $(1)$ and $(2)$ (i.e. $w\in w_0^\lambda{\cdot}R_{T^\lambda}$), then the relative position between $Y_w$ and any component of $Y_{w'}$ is given by $w^{-1}w_Pw'$. 
\end{Theorem}

\begin{proof}
    The relative position between the $\widetilde{P}$-orbits at $w$ and $w'$ is exactly given by $$r(\widetilde{P}wI/I,\widetilde{P}w'I/I)=w^{-1}w_Pw',$$ as the length of the product is the sum of the lengths. This follows from Proposition \ref{proprcellstr}.
	
	Consider a component $X$ of $Y_{w'}$. It thus follows that we have a bound on relative position given by $r(Y_w,X)\leq r(\widetilde{P}wI/I,\widetilde{P}w'I/I)=w^{-1}w_Pw'$.
	
	Now note that any component in $Y_{w'}$ contains a point of $\widetilde{W}$ in the $\widetilde{P}$-orbit of $w'$, as we have seen in the proof of Lemma \ref{fixpoints:P:orbits}. It thus follows for each component there exists an $x\in W_P$ such that the component contains $xw'$. Hence we have, that the relative position between $xw'$ and $x$ is given by $(w')^{-1}$.
	
	Now consider the relative position of $x$ with $Y_w$. For this it is enough to consider the $\widetilde{P}$-orbit at $w$. Further it is easy to see that the points in relative position $w_Pw$ to $x$ are precisely the preimage under the map $\widetilde{P}wI/I\rightarrow\widetilde{P}/I$ of the points in relative positions $w_P$ to the point $x$. Thus to check that the relative position of $x$ to $Y_w$ is $w_pw$ we just need to check that the relative position of $x$ to the Hessenberg variety $H^{W_w}_N$ is $w_P$.
	
	Recall that we have $l$ projections $\pi_i:\widetilde{P}/I\rightarrow SL_m/B$. Consider the preimage $U_i$ of the points in relative position $w_0\in S_m$ to $\pi_l(x)$. The points in relative position $w_P$ to $x$ is given by the intersection of all $U_i$.
	
	The Hessenberg variety for $w\in\mathfrak{F}$ is irreducible by Lemma \ref{connected:p:orbit}. It follows that if $U_i$ intersects non-trivially with $H^{W_w}_N$ for all $i$, then there is a point in $H^{W_w}_N$ in relative position $w_P$ to $x$, as non-empty open subsets of $H^{W_w}_N$ intersects in a non-empty subset. Recall that we contain the diagonal inclusion of $S_m\rightarrow (S_m)^l$ in the Hessenberg variety $H^{W_w}_N$. It is clear that for each $U_i$ there exists one of these points contained in $U_i$. It thus follows that the relative position of $Y_w$ to $x$ is $w^{-1}w_P$. Thus we get that the relative position of $Y_w$ and $xw'$ is $w^{-1}w_Pw'$, thus the relative position of $Y_w$ and any components of $Y_{w'}$ is at least $w^{-1}w_Pw'$.
	
	It follows from this lower bound and the above upper bound, that we have the relative position is given precisely by $w^{-1}w_Pw'$, as required.
\end{proof}

Further from the above computations we can also understand the exact number of components in the intersection of each $\widetilde{P}$-orbit.

\begin{Lemma}\label{lem:numberofcomp}
	Assume $w$ satisfies equations $(1)$, $(2)$ and $(3)_i$ precisely when $i\in I_w$ for some subset $I_w\subset[m-1]$. Consider the subgroup $S_{I_w}\subset S_m$ generated by the simple reflections $s_i$ for $i\in I_w$.
	
	Then $Y_w$ contains $\# S_m/S_{I_w}$ irreducible components.
\end{Lemma}

\begin{proof}
	As we have seen above we have a map $\widetilde{P}wI/I\cap\mcF l_N\rightarrow \widetilde{P}/I$ has image given by a Hessenberg variety and over that is a recursive affine space bundle.
	
	Further $Y_w^\circ$ is smooth and it thus follows that the irreducible components coincide with the connected components and these coincide with the connected components of the Hessenberg variety.
	
	Again recall we have several maps $\widetilde{P}/I\rightarrow SL_m/B$ and the image of the Hessenberg variety is contained in a Hessenberg variety corresponding to a regular semisimple map on $SL_m$.
	
	Note that the preimage under this map of each connected component is connected. This follows by Lemma \ref{connected:fibers} and the proof of Lemma \ref{connected:p:orbit}.
	
	Consider the parabolic $P_I\subset SL_m$ generated by $S_I$ and the Borel subgroup $B$. It follows from the description of the Hessenberg variety $H_h$ as described in \cite{boixeda2019fix,Boixedathesis} that $H_h^W$ is a subset of $S_m P_{I_w}/B$ and the intersection of $P_{I_w}/B$ with the Hessenberg variety is connected as it can be described as a product of Hessenberg varieties satisfying the conditions of connectedness.
	
	It follows from this that the number of connected components are precisely given by the number of connected components of $S_m P_{I_w}/B$. This is easily seen to be given by $\# S_m/S_{I_w}$. This gives a multinomial coefficient which describes the number of irreducible components of $Y_w$.
\end{proof}

We end this section by giving a description of every component as a component of the fundamental box after we apply the action of the $F_i$ as described in Section \ref{def:Fi}. We further identify the $Y_w$ for general $w$ in terms of this description.

\begin{Theorem}\label{lem:allcomp}
\begin{enumerate}
    \item Given any $A_w$ satisfying equations (1) and (2) (i.e. $w\in w_0^\lambda{\cdot}R_{T^\lambda}$), we have $$Y_w=\bigcup_{\sigma\in S_m}\prod_{i=1}^{m-1}F_{\sigma(i)}^{d_i+d_{i+1}+\ldots+d_{m-1}}(Y_{w'})$$ where $d_j=\Diff_j(w^{-1})-\lch_j(P(w^{-1}w_0^\lambda))$ for $j\in[m-1]$ and $w'\in \mfF$ satisfies $$w=w^{(k_1)}\phi^{k_1}(w^{(k_2)}\phi^{k_2}(\cdots w^{(k_\varepsilon)}\phi^{k_\varepsilon}(w')\cdots))$$ for $\{k_1,\ldots,k_\varepsilon\}=\{1^{d_1},2^{d_2},\ldots,(m-1)^{d_{m-1}}\}$ as a multi-set.
    \item All irreducible components of $\mcF l_N$ are given without repetition by $F^c(Y_w)$ for $w\in \mfF$, $c\in\mathbb{N}^m$ and $\min\{c_1,\ldots, c_m\}=0.$
\end{enumerate}
\end{Theorem}

\begin{proof}
	To begin consider the action on the points $\widetilde{W}\cap\mcF l_N$. The points of $\widetilde{W}$ are sent to themselves under the action of $F_i$ and thus it follows that the action preserves the set of elements of $\widetilde{W}$ that satisfy equation $(1)$.
	
	For an $m$-tuple $c$ recall that we define $F^c=F_1^{c_1}\ldots F_m^{c_m}$. We consider $c$ weakly decreasing $c_1\geq c_2\geq\ldots\geq c_m$.
	
	We begin by first claiming that $F^{\mathbb{1}_k}(w)=w^{(k)}\phi^k(w)$ where $\mathbb{1}_k$ is the $m$-tuple given by $k$ $1$'s followed by $m-k$ $0$'s. In fact, by direct computation,
	\begin{equation*}
	    \begin{split}
	      &\; s^{-k}f_1'^{-1}\ldots f_k'^{-1}\\
	     = & \;Ds^{-k}[1{+}m,2{+}m,\ldots,k{+}m,1{+}k,\ldots,m,1{+}2m,\ldots,k{+}2m,1{+}k{+}m,\ldots, 2m,\ldots,n] \\
	      = & \; D[1{+}m{-}k,2{+}m{-}k,\ldots,m,1,\ldots,m{-}k,1{+}2m{-}k,\ldots,2m,1{+}m,\ldots,2m{-}k,\ldots,n]\\
	      = & \; D \phi^{-k}(w^{(k)}).
	    \end{split}
	\end{equation*}
	where $D=\text{diag}\{1,\ldots,1,h_1^{-1},\ldots,h_k^{-1}\}.$ Hence $f_1'^{-1}\ldots f_k'^{-1}=\phi^k(D)w^{(k)}s^k  $ and the claim follows.
	
	
	Now we prove that $F^c({}^{w_P}IwI/I)\subset \widetilde{P}F^c(w)I/I$ for $c$  weakly decreasing and $w$ satisfying $(1)$ and $(2)$. Note that ${}^{w_P}IwI/I$ is an open subset of $\widetilde{P}wI/I$ and thus it will follow that the component $X$ of $Y_w$ containing $w$ satisfies that $F^c(Y_w)\subset Y_{F^c(w)}$.
	
	To check the statement note that it is enough to prove $\text{ad}(f')^c(U_\alpha)\subset\widetilde{P}$ for the root spaces $U_\alpha$ such that $w_P(\alpha)$ is positive and $w^{-1}(\alpha)$ is negative. To check this we just note that a root $e_i-e_j+k\delta$ can satisfy this only if $\overline{i}<\overline{j}$, where $\overline{i}\equiv i\text{ (mod m)}$ and $\overline{i}\in [m]$. Note that $\text{ad}(f')^c(U_\alpha)\subset\widetilde{P}$ if and only if $c_{\overline{i}}+\left\lfloor\frac{i}{m}\right\rfloor-c_{\overline{j}}-\left\lfloor\frac{j}{m}\right\rfloor\geq 0$. Note that as $U_\alpha\subset\widetilde{P}$, we have $\left\lfloor\frac{i}{m}\right\rfloor-\left\lfloor\frac{j}{m}\right\rfloor\geq 0$, thus the result follows as $\overline{i}<\overline{j}\Rightarrow c_{\overline{i}}\geq c_{\overline{j}}$ as $c$ is  weakly decreasing. 
	
	Consider the action of $S_m\hookrightarrow W_P$ the diagonal group. This acts on $N$ by sending it to a similar element, but with the regular semisimple $h$ exchanged by the action of $x\in S_m$ on it ${}^x h$. Note that these two are part of the family $\mcF l_{N}$ and so there is a monodromy action on the components on this space. Note that the $\widetilde{P}$-orbits are preserve, so there is an action on $Y_w$. Note that the action on the fix points $W_Pw\hookrightarrow\widetilde{P}wI/I$ is given by left multiplication. Note that the component of $Y_w$ are determined by the intersection with the points $S_m w\hookrightarrow W_p w$. Thus we get all the components of $Y_w$ are related by this action
	
	Using this we can check that for $\sigma\in S_m$ $\sigma(F^c(Y_w))=F^{\sigma(c)}(Y_w)$. Thus for $c$ weakly decreasing, we have $\cup_{\sigma\in S_m} F^{\sigma(c)}(Y_w)= Y_{F^c(w)}$ as we have an inclusion and both have the same number of components. From Corollary \ref{Cor:Fi}, we have the iterative expression $w=w^{(k_1)}\phi^{k_1}(w^{(k_2)}\phi^{k_2}(\cdots w^{(k_\varepsilon)}\phi^{k_\varepsilon}(w')\cdots))$ for any $w\in w_0^\lambda{\cdot}R_{T^\lambda}$, hence the first result is proved.
	
	
	Since $\mcF l_N=\bigcup_{w\in w_0^\lambda{\cdot}R_{T^\lambda}} Y_w$, we arrive at the second result by putting all irreducible components of $Y_w$'s in the first result together.
\end{proof}

\begin{Remark}
By Corollary \ref{corphim}, we know $\phi^m(w)\in \mfF$ if $w\in\mfF$. Hence it follows from the proof of Theorem \ref{lem:allcomp} that $F^{\mathbb{1}}(Y_w)=Y_{\phi^m(w)}$ (where $\mathbb{1}=\mathbb{1}_m$ in the vector of all 1's) is irreducible when $w\in \mfF$. This provides an alternative explanation of the condition $c\in\mathbb{N}^m$ and $\min\{c_1,\ldots, c_m\}=0$ in Theorem \ref{lem:allcomp} (2).
\end{Remark}

We end with a lemma to understand the relative position of the components with the above description.
\begin{Lemma}
\label{commonrotation}
    $r(F^{\mathbb{1}_k}(Y_{w'}),F^c (Y_w))=\phi^k(r(Y_{w'},F^{c-\mathbb{1}_k} (Y_w)))$ for any $w',w\in\mathfrak{F}$.
\end{Lemma}
\begin{proof}
    We will just prove that $r(F_i(x), F_i(y))=\phi(r(x,y))$, the result then follows, as the relative position of irreducible subsets is just the generic relative position of pairs of points.
    
    Note that $F_i$ is given by multiplication by $(f'_i)^{-1}$ followed by shifting the lattices by one. Note that this shifting can by understood as the following transformation $gI\mapsto gs^{-1}I$, where $s$ is the affine permutation introduced before Lemma \ref{lengthformula}.
    
    This is indeed well defined, because $sIs^{-1}=I$. Thus we have if $x=gI$, $y=g'I$, then $r(x,y)=r(I,g^{-1}g'I)$, where we have $g^{-1}g'I\in I r(I,g^{-1}g'I)I$. Note than $r(F_i(x),F_i(y))=r(I,sg^{-1}g's^{-1}I)$ and $sg^{-1}g's^{-1}I\in I sr(I,g^{-1}g'I)s^{-1}I$, but note that $sws^{-1}=\phi(w)$, hence the result follows as required. 
\end{proof}
\section{The case of $n=2$}
 The theorems about the Knuth equivalence classes in the Section 2 do not work for $n=2$, so we discuss the $n=2$ case separately in this section, where both geometry and combinatorics can be explicitly computed.
 
 For $n=2$, $\widetilde{S_2}$ is generated by $s_1=[2,1]$ and $s_0=[0,3]$ with relations $s_1^2=s_0^2=id$. There are two partitions $(2)$ and $(1,1)$ of size $2$. The two-sided cells $C_{(2)}=\{[1,2]\}$ and $C_{(1,1)}=\widetilde{S_2}\setminus C_{(2)}$. Explicit computation of affine matrix ball construction gives
 \begin{align*}
 & \Phi\left([2k+2,1-2k]=s_1(s_0s_1)^k\right)=\left(\;\raisebox{3pt}{\begin{ytableau}\overline{1 }\\ \overline{2}  \end{ytableau}}\,,\;\raisebox{3pt}{\begin{ytableau}\overline{1}\\ \overline{2}  \end{ytableau}}\,,\begin{pmatrix}\displaystyle -k\\k \end{pmatrix}\right), & k\ge 0;\\
 & \Phi\left([-1-2k,4+2k]=(s_1s_0)^{k+1}\right)=\left(\;\raisebox{3pt}{\begin{ytableau}\overline{1 }\\ \overline{2}  \end{ytableau}}\,,\;\raisebox{3pt}{\begin{ytableau}\overline{2}\\ \overline{1}  \end{ytableau}}\,,\begin{pmatrix}\displaystyle -k\\k \end{pmatrix}\right), & k\ge 0;\\
 & \Phi\left([2k+1,2-2k]=(s_0s_1)^k\right)=\left(\;\raisebox{3pt}{\begin{ytableau}\overline{2}\\ \overline{1}  \end{ytableau}}\,,\;\raisebox{3pt}{\begin{ytableau}\overline{1 }\\ \overline{2}  \end{ytableau}}\,,\begin{pmatrix}\displaystyle -k\\k \end{pmatrix}\right), & k\ge 1;\\
 & \Phi\left([-2k,3+2k]=s_0(s_1s_0)^k\right)=\left(\;\raisebox{3pt}{\begin{ytableau}\overline{2}\\ \overline{1}  \end{ytableau}}\,,\;\raisebox{3pt}{\begin{ytableau}\overline{2}\\ \overline{1}  \end{ytableau}}\,,\begin{pmatrix}\displaystyle -k\\k \end{pmatrix}\right), & k\ge 0.
 \end{align*}
 
 Now we consider the affine Springer fibers of two types. When $N=\begin{pmatrix}\displaystyle 0 &  1\\t & 0 \end{pmatrix}$, $\mcF l_N=\{I/I\}$, a singleton, and $r(I/I,I/I)=id\in C_{(2)}$.
 
 When $N=\begin{pmatrix}\displaystyle t &  0\\0 & -t \end{pmatrix}$, $\mcF l_N$ is an infinite chain of $\mathbb{P}^1$'s. Explicitly, it consists of irreducible components $\{C_k=F_1^k(C_0)\}_{k\in\mbZ}$, where $C_0=G(\mcO)/I=\{V.\in\mcF l\mid V_0=\mcO\langle e_1,e_2\rangle\}\simeq G/B$ and for $k\in\mbZ$,
 \begin{align*}
& C_{2k}=\{V.\in\mcF l\mid V_0=\mcO\langle t^{-k}e_1,t^ke_2\rangle\},\\
& C_{2k+1}=\{V.\in\mcF l\mid V_1=\mcO\langle t^{-k-1}e_1,t^ke_2\rangle\}.
 \end{align*}
 Moreover, each $C_k$ intersects only with $C_{k-1}$ and $C_{k+1}$ at precisely one flag.
 
 The entire affine Springer fiber is the union of $Y_w=\overline{G(\mcO)wI/I\cap \mcF l_N}$ for $w$ satisfying $w^{-1}(1)<w^{-1}(2)$. $Y_{id}=C_0$ and $Y_w=C_k\sqcup C_{-k}$ for $w=[1-k,2+k]^{-1}$ and $k\ge 1.$ Then $r(C_0,C_k)\le r(Y_{id},Y_w)\le s_1* s_1w=s_1w$ for $w=[1-k,2+k]^{-1}$. And by the explicit shape of $C_k$'s above, this upper bound can be reached by the relative position of certain coordinate flags in two components respectively. Hence $r(C_0,C_k)=s_1s_0s_1s_0\ldots$ where there are $k+1$ terms. Similarly $r(C_1,C_k)=s_0s_1s_0s_1\ldots$ where there are $k$ terms. 
 
 Since $\Lambda=\left\langle F_1F_2^{-1}=\begin{pmatrix}\displaystyle t^{-1} & 0\\0 & t \end{pmatrix}\right\rangle$ and $F_1F_2^{-1}(C_k)=C_{k+2}$, we have established the following commutative diagrams where $\theta$ and $\Theta$ are bijections:
 \begin{equation*}
 \begin{tikzcd}
 \Irr(\mathcal{F}l_{N})/\Lambda\arrow[dd,"\theta"] && \Irr(\mathcal{F}l_{N})\times_\Lambda \Irr(\mathcal{F}l_{N})\arrow[ll, "pr_i"']\arrow[rr, "r"]\arrow[dd,"\Theta"] && \widetilde{S_2}\\
 &&&&\\
 T((1,1)) && \Omega_{(1,1)}\arrow[ll,"pr_i"']\arrow[uurr,"\Psi"'] &&
 \end{tikzcd}
 \end{equation*}
 where $pr_i$, $i=1,2$ are the natural projection maps onto the first and second component and for all $k\in\mbZ$, $\theta(C_{2k})=\raisebox{3pt}{\begin{ytableau}\overline{1}\\ \overline{2}  \end{ytableau}}\,$, $\theta(C_{2k+1})=\raisebox{3pt}{\begin{ytableau}\overline{2}\\ \overline{1}  \end{ytableau}}$ and

 \begin{align*}
 & r(C_0,C_{2k})=(s_1s_0)^{|k|}s_1,\; & &  \Theta(C_0,C_{2k})=\left(\;\raisebox{3pt}{\begin{ytableau}\overline{1 }\\ \overline{2}  \end{ytableau}}\,,\;\raisebox{3pt}{\begin{ytableau}\overline{1 }\\ \overline{2}  \end{ytableau}}\,,\begin{pmatrix}\displaystyle -k\\k \end{pmatrix}\right),\\
  & r(C_0,C_{2k+1})=\left\{
  \begin{array}{ll}
  (s_1s_0)^{k+1}, & k\ge 0 \\
  (s_1s_0)^{-k}, & k\le-1\\
  \end{array} 
  \right. ,\;  & &  \Theta(C_0,C_{2k+1})=\left(\;\raisebox{3pt}{\begin{ytableau}\overline{1 }\\ \overline{2}  \end{ytableau}}\,,\;\raisebox{3pt}{\begin{ytableau}\overline{2 }\\ \overline{1}  \end{ytableau}}\,,\begin{pmatrix}\displaystyle -k\\k \end{pmatrix}\right),\\
  & r(C_1,C_{2k})=\left\{
  \begin{array}{ll}
  (s_0s_1)^{k}, & k\ge 1 \\
  (s_0s_1)^{-k+1}, & k\le 0\\
  \end{array} 
  \right. ,\;  &  &  \Theta(C_0,C_{2k+1})=\left(\;\raisebox{3pt}{\begin{ytableau}\overline{2}\\ \overline{1}  \end{ytableau}}\,,\;\raisebox{3pt}{\begin{ytableau}\overline{1}\\ \overline{2}  \end{ytableau}}\,,\begin{pmatrix}\displaystyle -k\\k \end{pmatrix}\right),\\
  & r(C_1,C_{2k+1})=(s_0s_1)^{|k|}s_0,\; & &  \Theta(C_0,C_{2k})=\left(\;\raisebox{3pt}{\begin{ytableau}\overline{2 }\\ \overline{1}  \end{ytableau}}\,,\;\raisebox{3pt}{\begin{ytableau}\overline{2 }\\ \overline{1}  \end{ytableau}}\,,\begin{pmatrix}\displaystyle -k\\k \end{pmatrix}\right).
 \end{align*}

\section{The main theorem}
In this final section, we explicitly compute the image of the relative positions between irreducible components of $\mcF l_N$ (obtained in Theorem \ref{lem:relativepositionforfundbox}) under the affine matrix ball construction. Then we conclude with the bijection from pairs of irreducible components modulo common translations to the triples $(P,Q,\rho)$ of rectangular shape and obtain an analogue of Steinberg and van Leeuven's result. The case of $n=2$ is explicitly computed in the previous section, so we focus on $n\ge3$ here. Recall from Theorem \ref{lem:allcomp} that all irreducible components of $\mcF l_N$ are given without repetition by $F^c(Y_w)$ for $w\in \mfF$, $c\in\mathbb{N}^m$ and $\min\{c_1,\ldots, c_m\}=0.$
\begin{Proposition}\label{proppremain}
	For $w',w''\in \mfF$, $c\in\mathbb{N}^m$ and $\min\{c_1,\ldots,c_m\}=0$, we have
	$$\Phi\left(r(Y_{w'},F^c(Y_{w''}))\right)=\left(\overline{(w')^{-1}}(T^\lambda),\overline{(w'')^{-1}}(T^\lambda)+|c|,\rho\right),$$where 
	\begin{align*}
	|c| = &\,c_1+\ldots+c_m,\\
	\rho = &\, \tilde{\rho}+s\left(\overline{(w')^{-1}}(T^\lambda)\right)-s\left(\overline{(w'')^{-1}}(T^\lambda)+|c|\right),\\
	\tilde{\rho}= & -\tilde{\rho}(w_0^\lambda w')+\tilde{\rho}(w_0^\lambda w'')-\left(\left\lfloor\frac{|c|}{n}\right\rfloor l+\delta^{\overline{|c|}}_1\left(\overline{(w'')^{-1}}\right)\right)\mathbb{1}+c^{\dom},
	\end{align*}
	and the function $\delta$ is defined to be $$ \delta^\alpha_i\left(\overline{u}\right)=\sum_{j=0}^{l-1}\mathbb{1}_{[n-\alpha+1,n]}\left(\overline{u}(i+jm)\right),\quad \alpha\in[0,n-1],i\in[m].$$
\end{Proposition}
\begin{proof}
	Let $c'_1\ge c'_2\ge\ldots\ge c'_m$ be the weakly decreasing rearrangement of $c_1,\ldots,c_m$, and $d_i=c'_i-c'_{i+1}$ for $i\in[m-1]$ and denote $\{k_1,\ldots,k_\varepsilon\}=\{1^{d_1},2^{d_2},\ldots,(m-1)^{d_{m-1}}\}$. So $k_1+\ldots+k_\varepsilon=|c|.$
	
	From Theorem \ref{lem:relativepositionforfundbox}, we know $r(Y_{w'},F^c(Y_{w''}))=(w')^{-1}w_0^\lambda w$ where $w\in w_0^\lambda{\cdot}R_{w_0^\lambda}$ such that $F^c(Y_{w''})\subset Y_w$. And from the proof of Proposition \ref{proprcellstr}, we have $(w')^{-1}w_0^\lambda w\sim_{\LKC}w_0^\lambda w$, $(w')^{-1}w_0^\lambda w$ and $(w')^{-1}w_0^\lambda$ are in the same right cell, so we know by Lemma \ref{lemAMBC} and Theorem \ref{thmknuth}:
	\begin{align*}
	& P((w')^{-1}w_0^\lambda w)=P((w')^{-1}w_0^\lambda)=\overline{(w')^{-1}}(T^\lambda),\\
	& Q((w')^{-1}w_0^\lambda w)=Q(w_0^\lambda w)=\overline{w^{-1}}(T^\lambda),\\
	& \rho((w')^{-1}w_0^\lambda w)=\rho((w')^{-1}w_0^\lambda )+\rho(w_0^\lambda w).
	\end{align*}
	 Since $w=w^{(k_1)}\phi^{k_1}(w^{(k_2)}\phi^{k_2}(\cdots w^{(k_\varepsilon)}\phi^{k_\varepsilon}(w'')\cdots))$, $$Q((w')^{-1}w_0^\lambda w)=\overline{w^{-1}}(T^\lambda)=\overline{(w'')^{-1}}(T^\lambda)+k_1+\ldots+k_s=\overline{(w'')^{-1}}(T^\lambda)+|c|.$$
	From Proposition \ref{propinvAMBC}, it suffices to show the following claim:
	\begin{equation}\label{eqmain}
	\tilde{\rho}(w^{-1}w_0^\lambda)-\tilde{\rho}((w'')^{-1}w_0^\lambda)=\left(\left\lfloor\frac{|c|}{n}\right\rfloor l+\delta^{\overline{|c|}}_1\left(\overline{(w'')^{-1}}\right)\right)\mathbb{1}+(-c)^{\dom}.
	\end{equation}
	
	Denote $u^{(\varepsilon)}=w''$ and $u^{(r)}=w^{(k_{r+1})}\phi^{k_{r+1}}(u^{(r+1)})$ for $r\in[0,\varepsilon-1]$. 
	
	Since $u^{(r)}\in w_0^\lambda{\cdot}R_{w_0^\lambda}$, we have $\rho_i((u^{(r)})^{-1}w_0^\lambda)=\sum_{j=0}^{l-1}\left\lceil\frac{(u^{(r)})^{-1}(i+jn)}{n}\right\rceil-l$ for any $r\in[0,\varepsilon]$ and $i\in[m]$.
	Explicitly, $(u^{(r)})^{-1}$ equals
	\begin{align*}
	\left[\right. & (u^{(r+1)})^{-1}(1+(l-1)m)-n+k,\ldots,(u^{(r+1)})^{-1}(k+(l-1)m)-n+k,\\
	&\qquad\qquad\qquad\qquad\qquad\qquad\qquad (u^{(r+1)})^{-1}(k+1)+k,\ldots, (u^{(r+1)})^{-1}(m)+k,\\
	& (u^{(r+1)})^{-1}(1)+k,\ldots,(u^{(r+1)})^{-1}(k)+k,\\
	&\qquad\qquad\qquad\qquad\qquad\quad\; (u^{(r+1)})^{-1}(k+1+m)+k,\ldots,(u^{(r+1)})^{-1}(2m)+k,\\
	&\ldots\\
	& (u^{(r+1)})^{-1}(1+(l-2)m)+k,\ldots,(u^{(r+1)})^{-1}(k+(l-2)m)+k,\\
	&\qquad\qquad\qquad\qquad\quad\quad (u^{(r+1)})^{-1}(k+1+(l-1)m)+k,\ldots,(u^{(r+1)})^{-1}(n)+k\left.\right].
	\end{align*}
Therefore,
	$$
	\rho_i((u^{(r)})^{-1}w_0^\lambda)=\rho_i((u^{(r+1)})^{-1}w_0^\lambda)-(\mathbb{1}_{k_{r+1}})_i+\delta_1^{k_{r+1}}\left(\overline{(u^{(r+1)})^{-1}}\right).
	$$
	Summing over $r\in[0,\varepsilon-1]$, we obtain:
\begin{equation}\label{eqrho}
\begin{split}
\rho(w^{-1}w_0^\lambda)=& \rho((w'')^{-1}w_0^\lambda)-\sum_{r=1}^\varepsilon\mathbb{1}_{k_r}+\sum_{r=1}^\varepsilon\delta^{k_r}\left(\overline{(u^{(r)})^{-1}}\right)\\
= & \rho((w'')^{-1}w_0^\lambda)-\sum_{r=1}^\varepsilon\mathbb{1}_{k_r}+\delta^{\overline{|c|}}\left(\overline{(w'')^{-1}}\right)+\left\lfloor\frac{|c|}{n}\right\rfloor l\mathbb{1}.
\end{split}
\end{equation}
Note that $\overline{|c|}$ in the $\delta$ function is taken to be in $[0,n-1]$.

Now we consider the entries of the symmetrized offset constant vectors.
\begin{equation}\label{eqs1}
\left(s\left(\overline{w^{-1}}(T^\lambda)+|c|\right)\right)_i=\sum_{j=0}^{l-1}\left(\left\lceil\frac{w^{-1}(i+jm)}{n}\right\rceil{-}\left\lceil\frac{w^{-1}(1+jm)}{n}\right\rceil\right){-}(d_1+\ldots+d_{i-1}),
\end{equation}
\begin{equation}\label{eqs2}
\left(s\left(\overline{(w'')^{-1}}(T^\lambda)+|c|\right)\right)_i=\sum_{j=0}^{l-1}\left(\left\lceil\frac{(w'')^{-1}(i+jm)}{n}\right\rceil{-}\left\lceil\frac{(w'')^{-1}(1+jm)}{n}\right\rceil\right).
\end{equation}
Also from the explicit form of $(u^{(r)})^{-1}$, we have:
$$\sum_{j=0}^{l-1}\left\lceil\frac{(u^{(r)})^{-1}(i+jm)}{n}\right\rceil=\sum_{j=0}^{l-1}\left\lceil\frac{(u^{(r+1)})^{-1}(i+jm)}{n}\right\rceil-(\mathbb{1}_{k_{r+1}})_i+\delta_i^{k_{r+1}}\left(\overline{(u^{(r+1)})^{-1}}\right).$$
So
\begin{align*}
& \sum_{j=0}^{l-1}\left(\left\lceil\frac{(u^{(r)})^{-1}(i+jm)}{n}\right\rceil-\left\lceil\frac{(u^{(r)})^{-1}(1+jm)}{n}\right\rceil\right)\\
=& \sum_{j=0}^{l-1}\left(\left\lceil\frac{(u^{(r+1)})^{-1}(i+jm)}{n}\right\rceil-\left\lceil\frac{(u^{(r+1)})^{-1}(1+jm)}{n}\right\rceil\right)-(\mathbb{1}_{k_{r+1}})_i+\delta_i^{k_{r+1}}\left(\overline{(u^{(r+1)})^{-1}}\right)\\
& +1-\delta_1^{k_{r+1}}\left(\overline{(u^{(r+1)})^{-1}}\right).
\end{align*}
Summing over $r\in[0,\varepsilon-1]$ again and rearranging terms, we get
\begin{equation}\label{eqceilcal}
\begin{split}
& \sum_{j=0}^{l-1}\left(\left\lceil\frac{w^{-1}(i{+}jm)}{n}\right\rceil{-}\left\lceil\frac{w^{-1}(1{+}jm)}{n}\right\rceil\right){-}\sum_{j=0}^{l-1}\left(\left\lceil\frac{(w'')^{-1}(i{+}jm)}{n}\right\rceil{-}\left\lceil\frac{(w'')^{-1}(1{+}jm)}{n}\right\rceil\right)\\
=&-\sum_{r=1}^\varepsilon(\mathbb{1}_{k_{r}})_i+\delta_i^{\overline{k_1+\ldots+k_\varepsilon}}\left(\overline{(w'')^{-1}}\right) +\left\lfloor\frac{k_1+\ldots+k_\varepsilon}{n}\right\rfloor l+\varepsilon-\delta_1^{\overline{k_1+\ldots+k_\varepsilon}}\left(\overline{(w'')^{-1}}\right).
\end{split}
\end{equation}
Equations \eqref{eqs1}, \eqref{eqs2}, \eqref{eqceilcal} together give
\begin{equation}\label{eqsdiff}
\begin{split}
& s(P(w^{-1}w_0^\lambda))-s(P((w'')^{-1}w_0^\lambda))\\
=& -\sum_{r=1}^\varepsilon\mathbb{1}_{k_r}+\delta^{\overline{k_1+\ldots+k_\varepsilon}}\left(\overline{(w'')^{-1}}\right)+\left(s-\delta_1^{\overline{k_1+\ldots+k_\varepsilon}}\left(\overline{(w'')^{-1}}\right)\right)\mathbb{1}-\begin{pmatrix}\displaystyle 0 \\ d_1\\d_1{+}d_2\\\ldots\\d_1{+}\ldots{+}d_{m-1} \end{pmatrix}.
\end{split}
\end{equation}
By definition of the $d_i$'s, we have
\begin{equation}\label{eqd}
\begin{pmatrix}\displaystyle 0 \\ d_1\\d_1{+}d_2\\\ldots\\d_1{+}\ldots{+}d_{m-1} \end{pmatrix}=\begin{pmatrix}\displaystyle c'_1{-}c'_1 \\ c'_1{-}c'_2\\c'_1{-}c'_3 \\\ldots\\c'_1{-}c'_m \end{pmatrix}=c'_1+(-c)^{\dom}=s+(-c)^{\dom}.
\end{equation}
Plugging \eqref{eqd} into \eqref{eqsdiff} and subtracting \eqref{eqsdiff} from \eqref{eqrho}, we arrived at the claim \eqref{eqmain}.
\end{proof}
We illustrate Proposition \ref{proppremain} using the following example.
\begin{Example}
	Let $n=6$, $\lambda=(2,2,2)$ and take $C_1=Y_{id}$ and $C_2=F_i^5F_j^2(Y_{w''})$ where $i,j\in[3],i\ne  j$ and $(w'')^{-1}=[0,1,4,3,5,8]\in\mfF^{-1}$ (which can be read from Figure \ref{LKC222} in the appendix). Then we know $\{k_1,\ldots,k_5\}=\{1,1,1,2,2\}$ and $C_2\subset Y_w$ where $$w=w^{(2)}\phi^2(w^{(2)}\phi^2(w^{(1)}\phi(w^{(1)}\phi(w^{(1)}\phi(w''))))).$$
	Direct computation gives $w^{-1}=[-8,2,11,-5,6,15]$. And we know $x\coloneqq r(C_1,C_2)=w_0^\lambda w=[12,2,-8,15,-5,5]$ and
	\begin{equation*}
	P(x)=\raisebox{12pt}{\begin{ytableau} \overline{1} & \overline{4}\\\overline{2} & \overline{5}\\\overline{3} & \overline{6}\end{ytableau}}\;,\;
	 Q(x)=\raisebox{12pt}{\begin{ytableau}\overline{1} & \overline{4}\\\overline{2} & \overline{6}\\\overline{3} & \overline{5}\end{ytableau}}\;,\;
	\rho(x)=\begin{pmatrix}\displaystyle -2 \\ 0\\2 \end{pmatrix}.
	\end{equation*}
	Now we use the formula in Proposition \ref{proppremain} to reproduce $\Phi(x)$.
	\begin{align*}
	P(x)&=\overline{(w')^{-1}}(T^\lambda)=T^\lambda,\\
	Q(x)&=\overline{(w'')^{-1}}(T^\lambda)+c_1+c_2+c_3=\raisebox{12pt}{\begin{ytableau}\overline{3} & \overline{6}\\\overline{1} & \overline{5}\\\overline{2} & \overline{4}\end{ytableau}}+7=\raisebox{12pt}{\begin{ytableau}\overline{1} & \overline{4}\\\overline{2} & \overline{6}\\\overline{3} & \overline{5}\end{ytableau}}\;,\\
	\rho(x)&= s\left(\,\raisebox{12pt}{\begin{ytableau}\overline{1} & \overline{4}\\\overline{2} & \overline{5}\\\overline{3} & \overline{6}\end{ytableau}}\,\right)-s\left(\,\raisebox{12pt}{\begin{ytableau}\overline{1} & \overline{4}\\\overline{2} & \overline{6}\\\overline{3} & \overline{5}\end{ytableau}}\,\right)-\tilde{\rho}(w_0^\lambda w')+\tilde{\rho}(w_0^\lambda w'')\\
	& \quad-\left(\left\lfloor\frac{c_1+c_2+c_3}{n}\right\rfloor l+\delta^{\overline{c_1+c_2+c_3}}_1\left(\overline{(w'')^{-1}}\right)\right)\mathbb{1}+c^{\dom}\\
	&= \begin{pmatrix}\displaystyle 0 \\ 0\\0 \end{pmatrix}-\begin{pmatrix}\displaystyle 0 \\ 0\\1 \end{pmatrix}-\begin{pmatrix}\displaystyle 0 \\ 0\\0 \end{pmatrix}+\begin{pmatrix}\displaystyle 1 \\ 1\\1 \end{pmatrix}-(2+1)\begin{pmatrix}\displaystyle 1 \\ 1\\1 \end{pmatrix}+\begin{pmatrix}\displaystyle 0 \\ 2\\5 \end{pmatrix}\\
	&= \begin{pmatrix}\displaystyle -2 \\ 0\\2 \end{pmatrix}.
	\end{align*}
\end{Example}

\begin{Theorem}\label{mainthmcal}
	Let $w',w''\in \mfF$, $c_1\ge 0,\ldots,c_m\ge 0$, $\min\{c_1,\ldots,c_m\}=0$, and $k\in[0,m-1]$. Then,
	$$\Phi\left(r(F^{\mathbb{1}_k}(Y_{w'}),F^c(Y_{w''}))\right)=\left(\overline{(w')^{-1}}(T^\lambda)+k,\overline{(w'')^{-1}}(T^\lambda)+|c|,\rho\right),$$where 
	\begin{align*}
	\rho=&\, \tilde{\rho}+s\left(\overline{(w')^{-1}}(T^\lambda)+k\right)-s\left(\overline{(w'')^{-1}}(T^\lambda)+|c|\right),\\
	\tilde{\rho}= & -\tilde{\rho}(w_0^\lambda w')+\tilde{\rho}(w_0^\lambda w'')-\left(\left\lfloor\frac{|c|}{n}\right\rfloor l+\delta^{\overline{|c|}}_1\left(\overline{(w'')^{-1}}\right)-\delta^k_1(\overline{(w')^{-1}})\right)\mathbb{1} +(c-\mathbb{1}_k)^{\dom},
	\end{align*}
and the $\delta$ function is defined to be $$ \delta^\alpha_i\left(\overline{u}\right)=\sum_{j=0}^{l-1}\mathbb{1}_{[n-\alpha+1,n]}\left(\overline{u}(i+jm)\right),\quad \alpha\in[0,n-1],i\in[m].$$
\end{Theorem}
\begin{proof}
	Denote $x\coloneqq r(F^{\mathbb{1}_k}(Y_{w'}),F^c(Y_{w''}))$ and $ z\coloneqq r(Y_{w'},F^{c-\mathbb{1}_k}(Y_{w''}))$. From Lemma \ref{commonrotation}, we know that $x=\phi^k(z)$. 
	
	Case 1: $\min\{c_1,\ldots,c_k\}\ge 1$.
	
	In this case we could apply Proposition \ref{proppremain} directly. Take $c'_1\ge c'_2\ge\ldots\ge c'_m$ be the weakly decreasing rearrangement of $c_1-1,\ldots,c_k-1,c_{k+1},\ldots,c_m$, and $d_i=c'_i-c'_{i+1}$ for $i\in[m-1]$ and denote $\{k_1,\ldots,k_\varepsilon\}=\{1^{d_1},2^{d_2},\ldots,(m-1)^{d_{m-1}}\}$. Note here $k_1+\ldots+k_\varepsilon=c_1+\ldots c_m-k.$ Then $F^{c-\mathbb{1}_k}(Y_{w''})\subset Y_w$ where $w=w^{(k_1)}\phi^{k_1}(w^{(k_2)}\phi^{k_2}(\cdots w^{(k_\varepsilon)}\phi^{k_\varepsilon}(w'')\cdots))$.

	 Hence by Lemma \ref{rotationAMBC}, we have:
	\begin{align*}
	& P(x)=P(z)+k=\overline{(w')^{-1}}(T^\lambda)+k,\\
	& Q(x)=Q(z)+k=\overline{(w'')^{-1}}(T^\lambda)+(c_1-1)+\ldots+(c_k-1)+c_{k+1}+\ldots+c_m+k\\
	&\qquad=\overline{(w'')^{-1}}(T^\lambda)+|c|,
	\end{align*}
	and
	\begin{equation}\label{eq1}
	\begin{split}
	& \rho(x)=\rho(z)+\delta^k\left(\overline{(w')^{-1}}\right)-\delta^k\left(\overline{(w'')^{-1}+k_1+\ldots+k_s}\right)\\
	&\quad\;\;\;=s\left(\overline{(w')^{-1}}(T^\lambda)\right)-s\left(\overline{(w'')^{-1}}(T^\lambda)+|c|-k\right)-\tilde{\rho}(w_0^\lambda w')+\tilde{\rho}(w_0^\lambda w'')\\
	&\qquad\quad-\left(\left\lfloor\frac{|c|-k}{n}\right\rfloor l+\delta^{\overline{|c|-k}}_1\left(\overline{(w'')^{-1}}\right)\right)\mathbb{1}+(c-\mathbb{1}_k)^{\dom}\\
	&\qquad\quad+\delta^k\left(\overline{(w')^{-1}}\right)-\delta^k\left(\overline{(w'')^{-1}+|c|-k}\right).
	\end{split}
	\end{equation}
	Since $$s_i\left(\overline{(w')^{-1}}(T^\lambda)+k\right)=s_i(P(w^{(k)}\phi^k(w')))-(\mathbb{1}-\mathbb{1}_k),$$
	we have
	\begin{equation}\label{eq2}
	s\left(\overline{(w')^{-1}}(T^\lambda)\right)=s\left(\overline{(w')^{-1}}(T^\lambda)+k\right)+\delta^k_1\left(\overline{(w')^{(-1)}}\right)\mathbb{1}-\delta^k\left(\overline{(w')^{(-1)}}\right).
	\end{equation}
	Similarly, there is
	\begin{equation}\label{eq3}
	\begin{split}
	& s\left(\overline{(w'')^{-1}}(T^\lambda)+|c|-k\right)=s\left(\overline{(w'')^{-1}}(T^\lambda)+|c|\right)\\
	& \qquad\qquad\qquad +\delta^k_1\left(\overline{(w'')^{-1}+|c|-k}\right)\mathbb{1}-\delta^k\left(\overline{(w'')^{-1}+|c|-k}\right).
	\end{split}
	\end{equation}
	Equations \eqref{eq1}, \eqref{eq2}, \eqref{eq3} together give
	\begin{align*}
	\tilde{\rho}(x)= & \delta^k_1\left(\overline{(w')^{(-1)}}\right)\mathbb{1}-\delta^k_1\left(\overline{(w'')^{-1}+|c|-k}\right)\mathbb{1} -\left(\left\lfloor\frac{|c|-k}{n}\right\rfloor l+\delta^{\overline{|c|-k}}_1\left(\overline{(w'')^{-1}}\right)\right)\mathbb{1}\\
	& -\tilde{\rho}(w_0^\lambda w')+\tilde{\rho}(w_0^\lambda w'')+(c-\mathbb{1}_k)^{\dom}\\
	= & -\left(\left\lfloor\frac{|c|}{n}\right\rfloor l+\delta^{\overline{|c|}}_1\left(\overline{(w'')^{-1}}\right)-\delta^k_1(\overline{(w')^{-1}})\right)\mathbb{1} -\tilde{\rho}(w_0^\lambda w')+\tilde{\rho}(w_0^\lambda w'')+(c-\mathbb{1}_k)^{\dom}.
	\end{align*}
	
	Case 2: $\min\{c_1,\ldots,c_k\}= 1$.
	
	In this case, we need to use Corollary \ref{corphim} and write the component as $$F^{c-\mathbb{1}_k}(Y_{w''})=F^{c+\mathbb{1}-\mathbb{1}_k}(Y_{\phi^{-m}(w'')}).$$
	Now we could apply Proposition \ref{proppremain}. Denote $y=\phi^{-m}(w'').$ and take $c'_1\ge c'_2\ge\ldots\ge c'_m$ be the weakly decreasing rearrangement of $c_1,\ldots,c_k,c_{k+1}+1,\ldots,c_m+1$, and $d_i=c'_i-c'_{i+1}$ for $i\in[m-1]$ and $\{k_1,\ldots,k_\varepsilon\}=\{1^{d_1},2^{d_2},\ldots,(m-1)^{d_{m-1}}\}$. Here $k_1+\ldots+k_\varepsilon=|c|+m-k.$ Similarly, $F^{c+\mathbb{1}-\mathbb{1}_k}(Y_{y})\subset Y_w$ where $w=w^{(k_1)}\phi^{k_1}(w^{(k_2)}\phi^{k_2}(\cdots w^{(k_\varepsilon)}\phi^{k_\varepsilon}(y)\cdots))$.
	Again by Lemma \ref{rotationAMBC}, we have
	\begin{align*}
	& P(x)=P(z)+k=\overline{(w')^{-1}}(T^\lambda)+k,\\
	& Q(x)=Q(z)+k=\overline{y^{-1}}(T^\lambda)+c_1+\ldots+c_k+(c_{k+1}+1)+\ldots+(c_m+1)+k\\
	&\qquad=\left(\overline{y^{-1}}(T^\lambda)+m\right)+|c|=\overline{(w'')^{-1}}(T^\lambda)+|c|,
	\end{align*}
	and
	\begin{equation}\label{eq4}
	\begin{split}
	& \rho(x)=\rho(z)+\delta^k\left(\overline{(w')^{-1}}\right)-\delta^k\left(\overline{y^{-1}+|c|+m-k}\right)\\
	&\quad\;\;\;=s\left(\overline{(w')^{-1}}(T^\lambda)\right)-s\left(\overline{y^{-1}}(T^\lambda)+|c|+m-k\right)-\tilde{\rho}(w_0^\lambda w')+\tilde{\rho}(w_0^\lambda y)\\
	&\qquad\quad-\left(\left\lfloor\frac{|c|+m-k}{n}\right\rfloor l+\delta^{\overline{|c|+m-k}}_1\left(\overline{y^{-1}}\right)\right)\mathbb{1}\\
	&\qquad\quad+(c+\mathbb{1}-\mathbb{1}_k)^{\dom}+\delta^k\left(\overline{(w')^{-1}}\right)-\delta^k\left(\overline{y^{-1}+|c|+m-k}\right).
	\end{split}
	\end{equation}
	Explicitly, \begin{equation*}
	\begin{split}
	y^{-1}=& \left[\right.(w'')^{-1}(1+m)-m,(w'')^{-1}(2+m)-m,\ldots,(w'')^{-1}(2m)-m,(w'')^{-1}(1+2m)-m,\\
	& \left.\ldots,(w'')^{-1}(n)-m,(w'')^{-1}(1)+n-m,\ldots,(w'')^{-1}(m)+n-m\right].
	\end{split}
	\end{equation*}
	Hence there is
	\begin{equation}\label{eq5}
	\begin{split}
	& s\left(\overline{y^{-1}}(T^\lambda)+|c|+m-k\right) =  s\left(\overline{(w'')^{-1}}(T^\lambda)+|c|-k\right) \\
	= & \;s\left(\overline{(w'')^{-1}}(T^\lambda)+|c|\right)+\delta_1^k\left(\overline{(w'')^{-1}+|c|-k}\right) -\delta^k\left(\overline{(w'')^{-1}+|c|-k}\right).
	\end{split}
	\end{equation}
	Combining equations \eqref{eq2}, \eqref{eq4}, \eqref{eq5}, we obtain
	\begin{equation}\label{eq6}
	\begin{split}
	\tilde{\rho}(x) = &\;-\tilde{\rho}(w_0^\lambda w')+\tilde{\rho}(w_0^\lambda y)+\delta_1^k\left(\overline{(w')^{-1}}\right)\mathbb{1}-\delta_1^k\left(\overline{(w'')^{-1}+|c|-k}\right)\mathbb{1}\\
	&\;-\left(\left\lfloor\frac{|c|+m-k}{n}\right\rfloor l+\delta^{\overline{|c|+m-k}}_1\left(\overline{(w'')^{-1}-m}\right)\right)\mathbb{1}+(c+\mathbb{1}-\mathbb{1}_k)^{\dom}\\
	=& \; -\tilde{\rho}(w_0^\lambda w')+\tilde{\rho}(w_0^\lambda y)+\delta_1^k\left(\overline{(w')^{-1}}\right)\mathbb{1}+(c+\mathbb{1}-\mathbb{1}_k)^{\dom}\\
	& \; -\left(\left\lfloor\frac{|c|+m}{n}\right\rfloor l+\delta^{\overline{|c|+m}}_1\left(\overline{(w'')^{-1}-m}\right)\right)\mathbb{1}.
	\end{split}
	\end{equation}
	Finally,
	\begin{equation}\label{eq7}
	\begin{split}
	\tilde{\rho}(w_0^\lambda y)= & \;-\tilde{\rho}(y^{-1}w_0^\lambda)=-\left(1-l+\sum_{j=0}^{l-1}\left\lceil\frac{(w'')^{-1}(1+jm)-m}{n}\right\rceil\right)\\
	=&\;-\tilde{\rho}((w'')^{-1}w_0^\lambda)-\left(1-\sum_{j=0}^{l-1}\mathbb{1}_{[1,m]}\left(\overline{(w'')^{-1}(1+jm)}\right)\right)\mathbb{1}.
	\end{split}  
	\end{equation}
	Let $\alpha=\overline{|c|+m}\in[0,n-1]$. Then
	\begin{equation}\label{eq8}
	\begin{split}
	& -\sum_{j=0}^{l-1}\mathbb{1}_{[1,m]}\left(\overline{(w'')^{-1}(1+jm)}\right)+\sum_{j=0}^{l-1}\mathbb{1}_{[n-\alpha+1,n]}\left(\overline{(w'')^{-1}(1+jm)}\right)\\
	=&\;\sum_{j=0}^{l-1}\left\{\begin{array}{ll}
	\mathbb{1}_{[n-\alpha+m+1,n]}\left(\overline{(w'')^{-1}(1+jm)}\right), & \alpha\ge m \\
	-\mathbb{1}_{[1,m-\alpha]}\left(\overline{(w'')^{-1}(1+jm)}\right), & \alpha<m \\
	\end{array} \right.\\
	=&\;\sum_{j=0}^{l-1}\left\{\begin{array}{ll}
	\mathbb{1}_{[n-\alpha+m+1,n]}\left(\overline{(w'')^{-1}(1+jm)}\right), & \alpha\ge m \\
	\mathbb{1}-\mathbb{1}_{[m-\alpha+1,n]}\left(\overline{(w'')^{-1}(1+jm)}\right), & \alpha<m \\
	\end{array} \right.
	\end{split}
	\end{equation}
	Combining \eqref{eq6}, \eqref{eq7} and \eqref{eq8} gives the same expression of $\rho(x)$ as in Case 1.
	\end{proof}

\begin{Example}
	Let $n=6,\lambda=(2,2,2)$ and $C_1=F_1F_2(Y_{w'}), C_2=F^2_2F_3^5(Y_{w''})$ where $(w')^{-1}=[-1,2,4,3,6,7]$ and $(w'')^{-1}=[0,1,4,3,5,8]$. It can be checked from Figure \ref{LKC222} in the appendix that $w',w''\in\mfF.$
	
	Then $y^{-1}=\phi^{-m}((w'')^{-1})=\phi^{-3}([0,1,4,3,5,8])=[0,2,5,3,4,7]$ and $$F_1^{-1}F^2_1F_3^5(Y_{w''})=F_2^2F_3^6(Y_y)\subset Y_w$$ where $w=w^{(2)}\phi^2(w^{(2)}\phi^2(w^{(1)}\phi(w^{(1)}\phi(w^{(1)}\phi(w^{(1)}\phi(w''))))))=[-10,4,13,-7,6,5]^{-1}$. So
	$$x\coloneqq r(C_1,C_2)=\phi^2((w')^{-1}w_0^\lambda w)=[15,2,-11,18,-7,4].$$
	Direct calculation by affine matrix ball construction gives
	$$\Phi(x)=\Phi([15,2,-11,18,-7,4])=\left(\;\raisebox{12pt}{\begin{ytableau}\overline{1} & \overline{5}\\\overline{2} & \overline{4}\\\overline{3} & \overline{6}\end{ytableau}}\;,\raisebox{12pt}{\begin{ytableau}\overline{1} & \overline{4}\\\overline{2} & \overline{6}\\\overline{3} & \overline{5}\end{ytableau}}\;,\begin{pmatrix}\displaystyle -3 \\ 0\\3 \end{pmatrix}\right).$$
	Now we use Theorem \ref{mainthmcal} to calculate the same data.
	\begin{align*}
	& P(x)=\overline{(w')^{-1}}(T^\lambda)+k=\raisebox{12pt}{\begin{ytableau}\overline{3} & \overline{5}\\\overline{2} & \overline{6}\\\overline{1} & \overline{4}\end{ytableau}}+2=\raisebox{12pt}{\begin{ytableau}\overline{1} & \overline{5}\\\overline{2} & \overline{4}\\\overline{3} & \overline{6}\end{ytableau}}\;,\\
	& Q(x)=\overline{(w'')^{-1}}(T^\lambda)+c_1+c_2+c_3=\raisebox{12pt}{\begin{ytableau}\overline{3} & \overline{6}\\\overline{1} & \overline{5}\\\overline{2} & \overline{4}\end{ytableau}}+7=\raisebox{12pt}{\begin{ytableau}\overline{1} & \overline{4}\\\overline{2} & \overline{6}\\\overline{3} & \overline{5}\end{ytableau}}\;,\\
	& \rho(x)= s\left(\raisebox{12pt}{\begin{ytableau}\overline{1} & \overline{5}\\\overline{2} & \overline{4}\\\overline{3} & \overline{6}\end{ytableau}}\right)-s\left(\raisebox{12pt}{\begin{ytableau}\overline{1} & \overline{4}\\\overline{2} & \overline{6}\\\overline{3} & \overline{5}\end{ytableau}}\right)-\tilde{\rho}(w_0^\lambda w')+\tilde{\rho}(w_0^\lambda w'')\\
	& \quad\;\;\; +\left(\delta^{k}_1\left(\overline{(w')^{-1}}\right)\right)\mathbb{1}-\left(\left\lfloor\frac{c_1+c_2+c_3}{n}\right\rfloor l+\delta^{\overline{c_1+c_2+c_3}}_1\left(\overline{(w'')^{-1}}\right)\right)\mathbb{1}+(c-\mathbb{1}_k)^{\dom}\\
	& \quad\;\;\;= \begin{pmatrix}\displaystyle 0 \\ 1\\1 \end{pmatrix}-\begin{pmatrix}\displaystyle 0 \\ 0\\1 \end{pmatrix}-\begin{pmatrix}\displaystyle 1 \\ 1\\1 \end{pmatrix}+\begin{pmatrix}\displaystyle 1 \\ 1\\1 \end{pmatrix}+\begin{pmatrix}\displaystyle 1 \\ 1\\1 \end{pmatrix}-(2+1)\begin{pmatrix}\displaystyle 1 \\ 1\\1 \end{pmatrix}+\begin{pmatrix}\displaystyle -1 \\ 1\\5 \end{pmatrix}\\
	& \quad\;\;\;= \begin{pmatrix}\displaystyle -3 \\ 0\\3 \end{pmatrix}.
	\end{align*}
\end{Example}

The main theorem of the paper is as follows:

\begin{Theorem}\label{mainthm}
	For $\lambda=(l^m)$, we have the following commutative diagrams:
	\begin{equation}
	\begin{tikzcd}
	\Irr(\mathcal{F}l_N)/\Lambda\arrow[dd,"\theta"] && \Irr(\mathcal{F}l_N)\times_\Lambda \Irr(\mathcal{F}l_N)\arrow[ll, "pr_i"']\arrow[rr, "r"]\arrow[dd,"\Theta"] && \widetilde{S_n}\\
	&&&&\\
	T(\lambda) && \Omega_{\lambda}\arrow[ll,"pr_i"']\arrow[uurr,"\Psi"'] &&
	\end{tikzcd}
	\end{equation}
	where 
	\begin{align*}
	&\theta(F^c(Y_w))=\overline{\phi^{m\alpha}(w^{-1})}(T^\lambda)+\beta, \text{ here }|c|=m\alpha+\beta,\alpha\in\mbZ,\beta\in[0,m-1],\\
	& \Omega_\lambda=\left\{(P,Q,\rho)\in\Omega\,\left|\, P,Q\in T(\lambda),\rho\in\mbZ^m\right.\right\},\\
	& \Theta\left(F^{\mathbb{1}_k}(Y_{w'}),F^c(Y_{w''})\right)=\left(\overline{(w')^{-1}}(T^\lambda)+k,\overline{(w'')^{-1}}(T^\lambda)+|c|,\rho\right)\text{ for } w',w''\in\mfF,\\
	& \rho= s\left(\overline{(w')^{-1}}(T^\lambda)+k\right)-s\left(\overline{(w'')^{-1}}(T^\lambda)+|c|\right)+ -\tilde{\rho}(w_0^\lambda w')+\tilde{\rho}(w_0^\lambda w'')\\
	&\qquad -\left(\left\lfloor\frac{|c|}{n}\right\rfloor l+\delta^{\overline{|c|}}_1\left(\overline{(w'')^{-1}}\right)-\delta^k_1(\overline{(w')^{-1}})\right)\mathbb{1} +(c-\mathbb{1}_k)^{\rev},\\
	& \delta^\alpha_i\left(\overline{u}\right)=\sum_{j=0}^{l-1}\mathbb{1}_{[n-\alpha+1,n]}\left(\overline{u}(i+jm)\right),\quad \alpha\in[0,n-1],i\in[m].
	\end{align*}
	And $pr_i$, $i=1,2$ are the natural projection maps onto the first and second component. Moreover, these maps satisfy:
	\begin{enumerate}
		\item The relative position $r$ maps onto the two-sided cell $C_\lambda$;
		\item $\theta$ and $\Theta$ are bijections;
		\item Given any $C\in\Irr(\mcF l_N)/\Lambda$, then $r(pr_1^{-1}(C))$ is a right cell in $C_\lambda$, and $r(pr_2^{-1}(C))$ is a left cell in $C_\lambda$.
	\end{enumerate}	 
\end{Theorem}
\begin{proof}
	For any $x\in C_\lambda$, we use Proposition \ref{proprcellstr} and consider the right cell it is contained in. Then $x=\phi^k((w')^{-1}w_0w)$ for some $k\in[0,m-1]$, $w'\in \mfF=w_0^\lambda{\cdot}\RKC_{w_0^\lambda}$ and $w\in w_0^\lambda{\cdot}R_{w_0^\lambda}$. Take $C$ to be any component containing in $Y_w$ and by Lemma \ref{commonrotation} we have $r(F^{\mathbb{1}_k}(Y_{w'}),F^{\mathbb{1}_k}(C))=x$, therefore proving the first claim.

	From the formulas in Theorem \ref{mainthm} and Corollary \ref{corconst}, we know the only factor that affects injectivity of $r$ is $(c-\mathbb{1}_k)^{\dom}$. Lemma \ref{lem:numberofcomp} and Proposition \ref{propnumcomb} show that $\Theta$ is a bijection. 
	
	Finally, the expression of the insertion (resp. recording) tabloid in the image of $\Theta$ only depends on the first (resp. second) component, therefore $r(pr_1^{-1}(C))$ is a right cell and $r(pr_2^{-1}(C))$ is a left cell.
\end{proof}
\begin{Remark}
	 From Lemma \ref{lem:numberofcomp} and Proposition \ref{propnumcomb}, the numbers of components in the $\tilde{P}$-orbit and the fibers of $\Psi$ are the same (both possess a Weyl group symmetry). Hence there are many ways to define the bijection $\Theta$ to make the diagram commute. When $(l^m)=(1^n)$, the definition of $\rho$ presented in Theorem \ref{mainthm} differs by a rotation with the one in \cite{yuganyou}. Whether there is a preferred choice reduces to the question of the geometric meaning of the weight vector $\rho$.
\end{Remark}

\begin{Corollary}
	When $\lambda=(l^m)$, the relative position map induces the following bijection:
	$$r:\Irr(\mcF l_N)\times_{\tilde{\Lambda}} \Irr(\mcF l_N)\rightarrow C_\lambda,$$
	where $\tilde{\Lambda}=\Lambda \rtimes S_m$. Explicitly, let
	$$
	C_1=F^{c'}(Y_{w'}), C_2=F^{c}(Y_{w}),
D_1=F^{d'}(Y_{u'}), D_2=F^{d}(Y_{u})$$
where $c',c,d',d\in\mathbb{Z}^m$ and $w',w,u',u\in \mfF.$
Then $(C_1,C_2)\sim(D_1,D_2)$ iff the following three criterions hold:
\begin{enumerate}
\item $\phi^{|c'|}(w')=\phi^{|d'|}(u'),$ $\phi^{|c|}(w)=\phi^{|d|}(u)$;
\item $m\mid (|c'|-|d'|),$ $m\mid(|c|-|d|);$
\item there exists $\sigma\in S_m$, such that $\sigma(c-c')-(d-d')$ is a constant vector.
\end{enumerate}
\end{Corollary}
\begin{proof}
	Let $|c'|=m\alpha'+\beta',|c|=m\alpha+\beta,|d'|=m\gamma'+\eta',|d|=m\gamma+\eta$ where $\alpha',\alpha,\gamma',\gamma\in\mbZ$ and $\beta',\beta,\eta',\eta\in[0,m-1].$ Then $$x\coloneqq r(C_1,C_2)=r\left(F^{\mathbb{1}_{\beta'}}(Y_{\phi^{m\alpha'}(w')}), F^{c-c'+\alpha'\mathbb{1}+\mathbb{1}_{\beta'}}(Y_{w})\right)=\phi^{\beta'}\left( r\left(Y_{\phi^{m\alpha'}(w')}, F^{c-c'+\alpha'\mathbb{1}}(Y_{w})\right)\right).$$
	Similarly,
	$$y\coloneqq r(D_1,D_2)=r\left(F^{\mathbb{1}_{\eta'}}(Y_{\phi^{m\gamma'}(u')}), F^{d-d'+\gamma'\mathbb{1}+\mathbb{1}_{\eta'}}(Y_{u})\right)=\phi^{\eta'}\left( r\left(Y_{\phi^{m\gamma'}(u')}, F^{d-d'+\gamma'\mathbb{1}}(Y_{u})\right)\right).$$
	Since $x=y$, we know $\beta'=\eta'$. Applying Proposition \ref{proppremain} and $P(\phi^{-\beta'}(x))=P(\phi^{-\eta'}(y))$, we have $\phi^{m\alpha'}(w')=\phi^{m\gamma'}(u')$. And $Q(\phi^{-\beta'}(x))=Q(\phi^{-\eta'}(y))$ indicates $\overline{w^{-1}}(T^\lambda)+|c|=\overline{u^{-1}}(T^\lambda)+|d|$, which is equivalent to $\beta=\eta$ and $\phi^{m\alpha}(w)=\phi^{m\gamma}(u)$. 
	
	Finally, we have $\rho(\phi^{-\beta'}(x))=\rho(\phi^{-\eta'}(y))$, which is equivalent to saying $F^{c-c'+\alpha'\mathbb{1}}(Y_w)$ and $F^{d-d'+\gamma'\mathbb{1}}(Y_u)$ are contained in the same $\widetilde{P}$-orbit. But $F^{c-c'+\alpha'\mathbb{1}}(Y_w)=F^{c-c'+\alpha'\mathbb{1}-\alpha\mathbb{1}}(Y_{\phi^{m\alpha}(w)})$, and $F^{d-d'+\gamma'\mathbb{1}}(Y_u)=F^{d-d'+\gamma'\mathbb{1}-\gamma\mathbb{1}}(Y_{\phi^{m\gamma}(u)})$. So there exists $\sigma\in S_m$, such that $\sigma(c-c'+\alpha'\mathbb{1}-\alpha\mathbb{1})=d-d'+\gamma'\mathbb{1}-\gamma\mathbb{1}$.
\end{proof}
A natural conjecture is the following:
\begin{Conjecture}
	For every partition $\lambda$ of $n$, there exists bijective maps $\theta_\lambda$ and $\Theta_\lambda$ that make the following two diagrams commutative:
	\begin{equation}
	\begin{tikzcd}
	\Irr(\mathcal{F}l_{N(\lambda)})/{\Lambda_{N(\lambda)}}\arrow[dd,"\theta_\lambda"] && \Irr(\mathcal{F}l_{N(\lambda)})\times_{\Lambda_{N(\lambda)}} \Irr(\mathcal{F}l_{N(\lambda)})\arrow[ll, "pr_i"']\arrow[rr, "r"]\arrow[dd,"\Theta_\lambda"] && \widetilde{S_n}\\
	&&&&\\
	T(\lambda) && \Omega_\lambda\arrow[ll,"pr_i"']\arrow[uurr,"\Psi"'] &&
	\end{tikzcd}
	\end{equation}
	where $N(\lambda)\in\mathfrak{g}(K)$ is a generic lift of a nilpotent element in $\mathfrak{g}$ of type $\lambda$ and $pr_i$, $i=1,2$ are the natural projection maps onto the first and second component.
\end{Conjecture}
\appendix
\section{Diagrams of left Knuth classes containing $w_0^\lambda$}
In this appendix, we show graphs of $\LKC_{w_0^\lambda}$ in case $\lambda=(2,2),(3,3)$ and $(2,2,2)$, where the edges corresponds to left Knuth moves.
\begin{figure}[!h]
	\label{LKC22}
	\begin{tikzcd}
	&\left[2,1,4,3\right]\arrow[dl, dash, "s_2\cdot"']\arrow[dr, dash, "s_0\cdot"]&\\
	\left[3,1,4,2\right]&&\left[2,0,5,3\right]
	\end{tikzcd}
	\caption{$\LKC_{w_0^\lambda}$ when $\lambda=(2,2)$}
\end{figure}
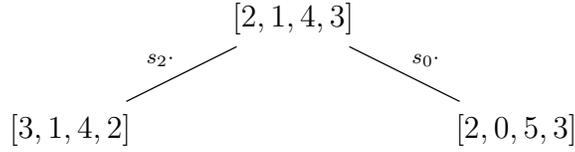

\begin{figure}
	\label{LKC33}
	\begin{tikzcd}
	&\left[2,1,4,3,6,5\right]\arrow[ddl, dash, "s_2\cdot"']\arrow[dd, dash, "s_4\cdot"']\arrow[ddr, dash, "s_0\cdot"]&\\
	&&\\
	\left[3,1,4,2,6,5\right]\arrow[dd, dash, "s_4\cdot"']\arrow[ddr, dash, "s_0\cdot", near end]&\left[2,1,5,3,6,4\right]\arrow[ddl, dash, "s_2\cdot"', near end]\arrow[ddr, dash, "s_0\cdot", near end]&\left[2,0,4,3,7,5\right]\arrow[ddl, dash, "s_2\cdot"', near end]\arrow[dd, dash, "s_4\cdot"]\\
	&&\\
	\left[3,1,5,2,6,4\right]\arrow[dd, dash, "s_3\cdot"']&\left[3,0,4,2,7,5\right]\arrow[dd, dash, "s_1\cdot"']&\left[2,0,5,3,7,4\right]\arrow[dd, dash, "s_5\cdot"]\\
	&&\\
	\left[4,1,5,2,6,3\right]&\left[3,0,4,1,8,5\right]&\left[2,-1,6,3,7,4\right]
	\end{tikzcd}
	\caption{$\LKC_{w_0^\lambda}$ when $\lambda=(3,3)$}
\end{figure}
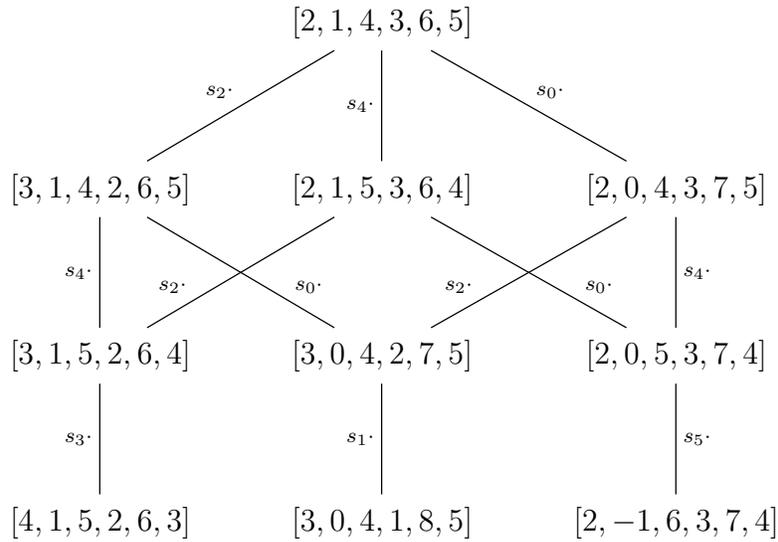

\begin{figure}
	\begin{turn}{90}
		\begin{adjustbox}{scale=0.82}
		\tiny{
		\begin{tikzcd}[column sep=tiny]
		&&&&&\left[3,2,1,6,5,4\right]\arrow[ddddll, dash, "s_3\cdot"']\arrow[ddddrr, dash, "s_0\cdot"]&&&&&\\
		&&&&&&&&&&\\
		&&&&&&&&&&\\
		&&&&&&&&&&\\
		&&&\left[4,2,1,6,5,3\right]\arrow[ddddll, dash, "s_2\cdot"']\arrow[dddd, dash, "s_4\cdot"']\arrow[ddddrr, dash, "s_0\cdot"]&&&&\left[3,2,0,7,5,4\right]\arrow[ddddll, dash, "s_3\cdot"']\arrow[dddd, dash, "s_5\cdot"']\arrow[ddddrr, dash, "s_1\cdot"]&&&\\
		&&&&&&&&&&\\
		&&&&&&&&&&\\
		&&&&&&&&&&\\
		&\left[4,3,1,6,5,2\right]\arrow[ddddl, dash, "s_0\cdot"']\arrow[ddddr, dash, "s_4\cdot"]&&\left[5,2,1,6,4,3\right]\arrow[ddddl, dash, "s_2\cdot"']\arrow[ddddr, dash, "s_0\cdot"]&&\left[4,2,0,7,5,3\right]&&\left[3,2,-1,7,6,4\right]\arrow[ddddl, dash, "s_3\cdot"']\arrow[ddddr, dash, "s_1\cdot"]&&\left[3,1,0,8,5,4\right]\arrow[ddddl, dash, "s_5\cdot"']\arrow[ddddr, dash, "s_3\cdot"]&\\
		&&&&&&&&&&\\
		&&&&&&&&&&\\
		&&&&&&&&&&\\
		\left[4,3,0,7,5,2\right]\arrow[dddd, dash, "s_1\cdot"']\arrow[ddddrr, dash, "s_4\cdot"]&&\left[5,3,1,6,4,2\right]&&\left[5,2,0,7,4,3\right]\arrow[dddd, dash, "s_5\cdot"]\arrow[ddddll, dash, "s_2\cdot"']&&\left[4,2,-1,7,6,3\right]\arrow[dddd, dash, "s_4\cdot"']\arrow[ddddrr, dash, "s_1\cdot"]&&\left[3,1,-1,8,6,4\right]&&\left[4,1,0,8,5,3\right]\arrow[ddddll, dash, "s_5\cdot"']\arrow[dddd, dash, "s_2\cdot"]\\
		&&&&&&&&&&\\
		&&&&&&&&&&\\
		&&&&&&&&&&\\
		\left[4,3,0,8,5,1\right]\arrow[ddddr, dash, "s_4\cdot"']\arrow[ddddddddrrrrr, dash, bend right=66, "s_2\cdot"', near end]&&\left[5,3,0,7,4,2\right]\arrow[ddddl, dash, "s_1\cdot"']\arrow[ddddr, dash, "s_5\cdot"]&&\left[6,2,-1,7,4,3\right]\arrow[ddddl, dash, "s_2\cdot"']\arrow[ddddr, dash, "s_4\cdot"]&&\left[5,2,-2,7,6,3\right]\arrow[ddddl, dash, "s_5\cdot"']\arrow[ddddr, dash, "s_1\cdot"]&&\left[4,1,-1,8,6,3\right]\arrow[ddddl, dash, "s_4\cdot"']\arrow[ddddr, dash, "s_2\cdot"]&&\left[4,1,0,9,5,2\right]\arrow[ddddl, dash, "s_5\cdot"]\arrow[ddddddddlllll, dash, bend left=66, "s_1\cdot", near end]\\
		&&&&&&&&&&\\
		&&&&&&&&&&\\
		&&&&&&&&&&\\
		&\left[5,3,0,8,4,1\right]\arrow[ddddr, dash, "s_5\cdot"']&&\left[6,3,-1,7,4,2\right]\arrow[ddddl, dash, "s_1\cdot"]&&\left[6,2,-2,7,5,3\right]&&\left[5,1,-2,8,6,3\right]\arrow[ddddr, dash, "s_2\cdot"']&&\left[4,1,-1,9,6,2\right]\arrow[ddddl, dash, "s_4\cdot"]&\\
		&&&&&&&&&&\\
		&&&&&&&&&&\\
		&&&&&&&&&&\\
		&&\left[6,3,-1,8,4,1\right]\arrow[dddd, dash, "s_0\cdot"']&&&\left[4,2,0,9,5,1\right]&&&\left[5,1,-2,9,6,2\right]\arrow[dddd, dash, "s_3\cdot"]&&\\
		&&&&&&&&&&\\
		&&&&&&&&&&\\
		&&&&&&&&&&\\
		&&\left[7,3,-1,8,4,0\right]&&&&&&\left[5,1,-3,10,6,2\right]&&
		\end{tikzcd}}
	\end{adjustbox}
	\end{turn}
\caption{$\LKC_{w_0^\lambda}$ when $\lambda=(2,2,2)$}
\label{LKC222}
\end{figure}
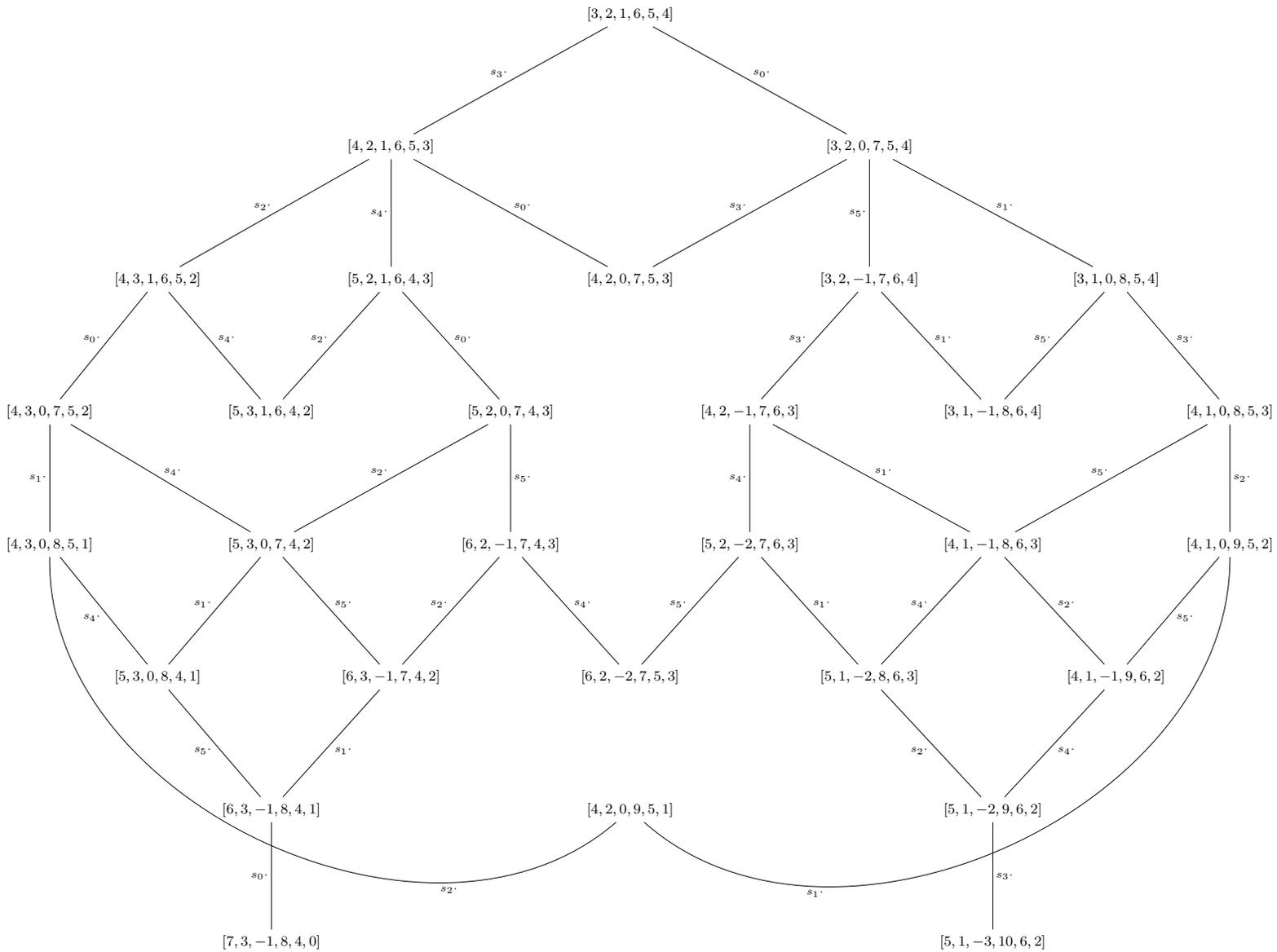

\bibliography{main}
\bibliographystyle{plain}
\end{document}